\documentclass[11pt,leqno,
,a4paper]{article} 
\pdfoutput=1 
\usepackage{amsmath}
\usepackage{amssymb}
\usepackage{amsthm}
\usepackage{verbatim}
\usepackage{enumerate}
\usepackage{color}
\usepackage{blkarray}
\usepackage{bm}
\usepackage{todonotes}
\usepackage{comment}
\usepackage{placeins}
\usepackage{comment}
\usepackage{color}
\usepackage[affil-it]{authblk}
\usepackage{mathtools}
\mathtoolsset{showonlyrefs,showmanualtags}
\numberwithin{equation}{section}
\usepackage[flushmargin]{footmisc}

\usepackage{tikz}
\usetikzlibrary{cd}
\usepackage{graphicx,amscd,mathrsfs}
\usetikzlibrary{decorations.pathreplacing,decorations.pathmorphing,decorations.markings}
\usepackage{tikz-cd}

\usepackage[top=30truemm,bottom=30truemm,left=30truemm,right=30truemm]{geometry}

\usepackage{hyperref}
\definecolor{cmykred}{cmyk}{0,0.8,0.5,0}
\definecolor{cmykgreen}{cmyk}{0.692,0,0.287,0.439}
\hypersetup{colorlinks=true,urlcolor=blue,citecolor=cmykred,linkcolor=blue}

\theoremstyle{definition}
\newtheorem{definition}{Definition}[section]
\newtheorem{remark}[definition]{Remark}
\theoremstyle{plain}
\newtheorem{theorem}[definition]{Theorem}
\newtheorem{lemma}[definition]{Lemma}
\newtheorem{proposition}[definition]{Proposition}
\newtheorem{corollary}[definition]{Corollary}
\theoremstyle{remark}

\def\p{\partial}

\def\re{\mathrm{\mathop{Re}}}

\def\bbC{\mathbb{C}}
\def\bbP{\mathbb{P}}
\def\bbR{\mathbb{R}}
\def\bbZ{\mathbb{Z}}

\def\calV{\mathcal{V}}
\def\calY{\mathcal{Y}}
\def\bbC{\mathbb{C}}

\newcommand{\two}{I\hspace{-1.2pt}I}
\newcommand{\three}{I\hspace{-1.2pt}I\hspace{-1.2pt}I}
\newcommand{\four}{I\hspace{-1.2pt}V}
\newcommand{\six}{V\hspace{-1.2pt}I}
\newcommand{\seven}{V\hspace{-1.2pt}I\hspace{-1.2pt}I}
\newcommand{\eight}{V\hspace{-1.2pt}I\hspace{-1.2pt}I \hspace{-1.2pt}I}
\newcommand{\nine}{I\hspace{-1.2pt}X}
\newcommand{\eleven}{X\hspace{-1.2pt}I}
\newcommand{\twelve}{X\hspace{-1.2pt}I\hspace{-1.2pt}I}

\title{Asymptotic analysis of the confluent Dotsenko--Fateev system}
\author{
	Shunya Adachi
}
\date{}

\allowdisplaybreaks
\begin{document}

\maketitle

\begingroup
\renewcommand{\thefootnote}{}
\footnotetext{%
\setlength{\leftskip}{1em}%

}
\endgroup

\begin{abstract}
The Dotsenko--Fateev equation is a third-order linear ordinary differential equation studied in two-dimensional conformal field theory. 
We study the confluent Dotsenko--Fateev system, a rank-three system of linear ordinary differential equations obtained by confluence from a first-order system equivalent to the Dotsenko--Fateev equation.
It has an irregular singular point at infinity, and its solutions admit double integral representations of confluent Selberg type. 
Using these representations, we construct sectorial fundamental solution matrices around infinity, establish their asymptotic expansions, and determine the associated Stokes matrices explicitly. 
We further study three aspects of the confluent Dotsenko--Fateev system:  its equivalence to the degree-two polynomial reduction of the Schr\"{o}dinger equation for the quantum Painlev\'{e} equation of type V, its realization as a restriction of a confluent version of the linear Pfaffian system associated with Appell's hypergeometric function $F_4$, and its reduction to the trivial rank-one equation via integral transforms and a quadratic change of the independent variable. 
\end{abstract}

\section{Introduction}
In the context of two-dimensional conformal field theory, Dotsenko and Fateev \cite{DF1984} studied the Selberg-type double integral 
\begin{equation}\label{eq:DFintscalar}
	u(x)=\int_{C} (s_1-s_2)^g \prod_{i=1}^2 s_i^a (s_i-1)^b (s_i-x)^c\,\frac{ds_1\wedge ds_2}{s_1 s_2}.
\end{equation}
They showed that this integral satisfies a third-order Fuchsian linear ordinary differential equation, which we call the \emph{Dotsenko--Fateev equation}. 
This equation can be regarded as a higher-order analogue of the Gauss hypergeometric differential equation.
Dotsenko and Fateev studied the connection problem, monodromy representation, and monodromy invariant Hermitian form of the Dotsenko--Fateev equation using the integral \eqref{eq:DFintscalar}. 
More recently, Mimachi \cite{Mimachi2024} gave a condition for the (ir)reducibility of the monodromy representation and studied subsystems in certain reducible cases.

In this paper, we consider the following rank-three system obtained by confluence from a Fuchsian system equivalent to the Dotsenko--Fateev equation:
	\begin{equation}\label{eq:cDFeqintro}
	\frac{dV}{dx}=\left(T+\frac{A}{x}\right)V,
\end{equation}
where
\begin{equation}\label{eq:TandAIntro}
	T=\begin{pmatrix}
	0 & & \\
	& 1 & \\
	& & 2
	\end{pmatrix},
	\quad 
	A=\begin{pmatrix}
	2a+2b+g & -b &~0 \\
	2a+2b+g & a+g&~-2b-g \\
	0 & a+b+g &~ -2b-g
	\end{pmatrix}. 
\end{equation}
We call this system the \emph{confluent Dotsenko--Fateev system}.
It has a regular singular point at $x=0$ and an irregular singular point of Poincar\'{e} rank $1$ at $x=\infty$.
In terms of its singularity structure, it can be regarded as a higher-rank analogue of Kummer's confluent hypergeometric equation.
The main purpose of this paper is to study the behavior of solutions near the irregular singular point.
This system has integral representations of solutions of the form 
	\begin{equation}\label{eq:cDFintintro}
	V(x)=\int_{\Delta}(t_1-t_2)^{g}\prod_{i=1}^{2}t_i^{-(a+b+g)}(1-t_i)^{b}e^{xt_i}	\vec{\varphi},
	\end{equation}
where $\Delta$ is a suitable cycle and $\vec{\varphi}$ is a $3$-vector of $2$-forms. 
Since $T$ has pairwise distinct eigenvalues, the system \eqref{eq:cDFeqintro} admits a unique formal fundamental solution matrix at $x=\infty$ of the form
\begin{equation}\label{eq:FSMintro}
  \widehat{\calY}(x)=\widehat{F}(x)x^\Lambda e^{Tx},
  \qquad
  \widehat{F}(x)=I+\sum_{n=1}^{\infty}F_nx^{-n},
\end{equation}
where $\Lambda=\operatorname{diag}(A)=\operatorname{diag}[2a+2b+g,\ a+g,\ -2b-g]$. 
Using the integral representation \eqref{eq:cDFintintro}, we construct sectorial fundamental solution matrices and show that their normalized forms admit $\widehat{\calY}(x)$ as their asymptotic expansion at infinity (Theorem~\ref{thm:asymptoticexpansion}). 
We then adapt Haraoka's contour-deformation method \cite{Haraoka2024} to determine the associated Stokes matrices explicitly (Theorem~\ref{thm:StokesMat}).

For generalized confluent hypergeometric equations,
explicit Stokes matrices have been obtained by
Duval and Mitschi~\cite{DuvalMitschi1989} and Hien~\cite{Hien2022}.
However, the confluent Dotsenko--Fateev system does not, in general,
belong to this class.
In a different direction, Glutsyuk and Sabot~\cite{GlutsyukSabot2010} computed Stokes matrices of ordinary differential equations satisfied by exponential hypergeometric integrals associated with hyperplane arrangements in general position. 
But their result does not apply directly to the present integral
\eqref{eq:cDFintintro}, since the underlying hyperplane
arrangement is not in general position: it has parallel hyperplanes and triple intersections.

We further discuss three aspects of the confluent Dotsenko--Fateev system: an application to the Schr\"{o}dinger equation for quantum Painlev\'{e} V, the compatibility of confluence and restriction for Appell's $F_4$ system, and a reduction to the trivial rank-one equation via integral transforms and a quadratic change of the independent variable.
First, we show that the confluent Dotsenko--Fateev system \eqref{eq:cDFeqintro} is equivalent to the degree-two polynomial reduction of the Schr\"{o}dinger equation for quantum Painlev\'{e} V studied by Nagoya \cite{Nagoya2011} (Proposition~\ref{prop:NagoyacDF}). 
Nagoya also constructed polynomial solutions of this Schr\"{o}dinger equation using integral representations.
We show that, in the degree-two case, sectorial bases of these polynomial solutions can be expressed in terms of the fundamental solution matrices constructed in Section~\ref{sec:asymptoticanalysis}. 
The Stokes matrices for these sectorial bases then follow directly from Theorem~\ref{thm:StokesMat} (Corollary~\ref{cor:StokesSch}).

Second, we examine the relation to the linear Pfaffian system associated with Appell's hypergeometric function $F_4$ in two variables.
The work of Kato \cite{Kato1995} and Haraoka \cite{Haraoka2013} shows that the Dotsenko--Fateev equation arises by restricting a pull-back of Appell's $F_4$ system to a component of its singular locus. 
More recently, Haraoka \cite{Haraoka2024} studied a confluent $F_4$ system
and explicitly determined its Stokes multipliers.
We show that, in the present case, the restriction to the singular locus is compatible with the confluence.
More precisely, we show that, after a scalar gauge transformation, the restriction of Haraoka's confluent $F_4$ system to the corresponding component of its singular locus is equivalent to the confluent Dotsenko--Fateev system \eqref{eq:cDFeqintro} (Proposition~\ref{prop:restcF4}). 
However, the Stokes matrices of the restricted system do not follow immediately from Haraoka's result (Remark~\ref{rem:AdaHara}).

Finally, we show that, for generic parameters, the confluent Dotsenko--Fateev system \eqref{eq:cDFeqintro} can be reduced to the trivial rank-one equation by combining integral transforms with a quadratic change of the independent variable (Theorem~\ref{thm:reduction}).
The integral transforms used in this reduction are middle convolution and the middle Laplace transform.
Middle convolution is a basic operation in Katz's theory of Fuchsian systems and is constructed from Euler-type integral transforms
(cf.~Dettweiler--Reiter~\cite{DR2007}).
The middle Laplace transform, introduced by the author for linear Pfaffian systems with irregular singularities \cite{adachi2025arXiv}, formulates the Laplace transform
from the viewpoint of Katz's theory. 
The quadratic change of the independent variable plays an essential role in this reduction: it changes the rigidity index from $0$ to $2$, whereas all other steps in this reduction leave the rigidity index unchanged.
For related work on the Dotsenko--Fateev equation from
the viewpoint of middle convolution, we refer to
Ebisu et al.~\cite{EbisuEtAl2023} and Haraoka et al.~\cite{arXiv:2111.11192}.

\smallskip

A natural direction for further study is to extend the present analysis to the case of multiple integrals with more integration variables. 
The Dotsenko--Fateev integral \eqref{eq:DFintscalar} corresponds to the case $m=2$ of the Selberg-type integral
	\begin{equation}\label{eq:SelbergIntegral}
	\int_{C}\prod_{1\le i<j\le m}(s_i-s_j)^g
	\prod_{k=1}^{m}s_k^a(s_k-1)^b(s_k-x)^c\,\frac{ds_1\wedge ds_2\wedge\cdots\wedge ds_m}{s_1s_2\cdots s_m}.
	\end{equation}
Mimachi \cite{Mimachi1993} explicitly described the rank-$(m+1)$ Fuchsian system with singular points at $x=0$, $1$ and $\infty$ associated with \eqref{eq:SelbergIntegral}.
He subsequently obtained explicit connection formulas for the associated Selberg-type integrals \cite{Mimachi2013}.
The confluence described in Section~\ref{sec:conf}, followed by the change of variables $t_k=1/s_k$, leads formally to confluent Selberg-type integrals of the form
	\begin{equation}\label{eq:confSelberg}
	\int_{\Delta}\prod_{1\le i<j\le m}(t_i-t_j)^g
	\prod_{k=1}^{m}t_k^{-(a+b+(m-1)g)}(1-t_k)^be^{xt_k}\,\frac{dt_1\wedge dt_2\wedge\cdots\wedge dt_m}{t_1t_2\cdots t_m}.
	\end{equation}
The corresponding first-order system is of rank $(m+1)$ and has a regular singular point at $x=0$ and an irregular singular point at $x=\infty$.
For these systems, a natural problem is to construct sectorial fundamental solution matrices at infinity and determine the associated Stokes matrices explicitly using the integral representation \eqref{eq:confSelberg}.
It would also be interesting to investigate whether the three aspects discussed above admit natural extensions to this higher-rank setting.

\smallskip

This paper is organized as follows. 
In Section~\ref{sec:conf}, we derive the confluent Dotsenko--Fateev system and integral representations of its solutions by confluence. 
Using these integral representations, we construct sectorial fundamental solution matrices at infinity and determine the associated Stokes matrices in Section~\ref{sec:asymptoticanalysis}.
In Section~\ref{sec:threeaspects}, we study the three aspects outlined above: an application to the Schr\"{o}dinger equation associated with quantum Painlev\'{e} V, the relation to the confluent $F_4$ system through restriction, and the reduction to the trivial rank-one equation. 

\section{Confluence and integral representations}
\label{sec:conf}

Let us derive the confluent Dotsenko--Fateev system and its integral representation by confluence. 
We first recall a rank-three Fuchsian system associated with the Dotsenko--Fateev equation (cf.~Haraoka~\cite[Section~2]{Haraoka2013}):
\begin{equation}\label{eq:DFeq}
\frac{dU}{dx}=\left(\frac{A_1}{x}+\frac{A_2}{x-1}\right)U,
\end{equation}
where
\begin{equation}
\begin{aligned}
A_1&=\begin{pmatrix}
2a+2c+g & 0 & b \\
0 & 0 & 0 \\
0 & 2b+g & a+c
\end{pmatrix}
\quad (\sigma(A_1)=\{0, a+c, 2a+2c+g\}), 
\\
A_2&=\begin{pmatrix}
0 & 0 & 0 \\
0 & 2b+2c+g & a \\
2a+g & 0 & b+c 
\end{pmatrix}
\quad (\sigma(A_2)=\{0, b+c, 2b+2c+g\}), \\
A_{\infty}&:=-(A_1+A_2)
\quad (\sigma(A_{\infty})=\{-2c, -(2a+2b+2c+g), -(a+b+2c+g)\}).
\end{aligned}
\end{equation}
Here the unknown vector $U(x)$ is given by 
\begin{equation}\label{eq:DFint}
U(x)=\int_{C} (s_1-s_2)^g \prod_{i=1}^2 s_i^a (s_i-1)^b (s_i-x)^c\,\vec{\eta},
\end{equation}
where
\begin{equation}\label{eq:eta}
	\begin{aligned}
	\vec{\eta}
	=\left(
	\frac{ds_1\wedge ds_2}{s_1 s_2}, 
	\frac{ds_1\wedge ds_2}{(s_1-1)(s_2-1)}, 
	\frac{ds_1\wedge ds_2}{s_1(s_2-1)}+\frac{ds_1\wedge ds_2}{s_2(s_1-1)}
	\right)^T
	\end{aligned}.
\end{equation}
Note that the first component of $U(x)$ coincides with \eqref{eq:DFintscalar}. 
In the Dotsenko--Fateev system \eqref{eq:DFeq}, we make the substitution
\begin{equation}
(a,x) \mapsto \left(a-c, -\frac{x}{c}\right)
\end{equation}
and then take the limit $c\to \infty$. 
Then the system \eqref{eq:DFeq} turns into
\begin{equation}\label{eq:cDFeq0}
	\frac{dV_0}{dx}=\left(T_0+\frac{A_1}{x}\right)V_0,
	\quad
	A_1=\begin{pmatrix}
	2a+g & 0 & b \\
	0 & 0 & 0 \\
	0 & 2b+g & a
	\end{pmatrix},
	\quad 
	T_0=\begin{pmatrix}
	0 & 0 & 0\\
	0 & 2 & -1 \\
	-2 & 0 & 1
	\end{pmatrix}.
\end{equation}
We refer to this equation as \emph{the confluent Dotsenko--Fateev system}. 

By the same procedure, the integrand of \eqref{eq:DFint} changes as follows: 
\begin{equation}
\begin{aligned}
	(s_1-s_2)^g \prod_{i=1}^2 s_i^a (s_i-1)^b (s_i-x)^c &\mapsto 
	(s_1-s_2)^g \prod_{i=1}^2 s_i^{a-c} (s_i-1)^b \left(s_i+\frac{x}{c}\right)^c \\
	&=(s_1-s_2)^g \prod_{i=1}^2 s_i^{a} (s_i-1)^b \left(1+\frac{x}{cs_i}\right)^c \\
	&\to (s_1-s_2)^g \prod_{i=1}^2 s_i^{a} (s_i-1)^b e^{\frac{x}{s_i}},
\end{aligned}
\end{equation}
and the vector $\vec{\eta}$ is unchanged. 
This formally leads to the following integral representation
of solutions of \eqref{eq:cDFeq0}:
\begin{equation}\label{eq:cDFint0}
	V_0(x)=\int_{\Delta}(s_1-s_2)^g \prod_{l=1}^2 s_l^{a} (s_l-1)^b e^{\frac{x}{s_l}}\,\vec{\eta}
\end{equation}
with a suitable cycle $\Delta$, which will be specified later.
Here the vector of cocycles $\vec{\eta}$ is given by \eqref{eq:eta}. 
To simplify the analysis, we consider the gauge transformation
\begin{equation}\label{eq:gauge}
	V_0=PV, \quad 
	P=\begin{pmatrix}
	-1 & 0 & 0 \\
	-1 & 1 & -1 \\
	-2 & 1 & 0
	\end{pmatrix}.
\end{equation}
Then the system \eqref{eq:cDFeq0} is changed to
\begin{equation}\label{eq:cDFeq}
	\frac{dV}{dx}=\left(T+\frac{A}{x}\right)V,
\end{equation}
where
\begin{equation}\label{eq:TandA}
	T=\begin{pmatrix}
	0 & & \\
	& 1 & \\
	& & 2
	\end{pmatrix},
	\quad 
	A=\begin{pmatrix}
	2a+2b+g & -b &~0 \\
	2a+2b+g & a+g&~-2b-g \\
	0 & a+b+g &~ -2b-g
	\end{pmatrix}. 
\end{equation}
Applying the gauge transformation \eqref{eq:gauge} and the change of integral variables $(s_1,s_2)=(1/t_2,1/t_1)$ to the integral \eqref{eq:cDFint0}, we have
\begin{equation}\label{eq:cDFint}
	V(x)=\int_{\Delta}(t_1-t_2)^{g}\prod_{l=1}^{2}t_l^{-(a+b+g)}(1-t_l)^{b}e^{xt_l}	\vec{\varphi}, 
\end{equation}
where $\vec{\varphi}=P^{-1}\vec{\eta}=(\varphi_1,\varphi_2,\varphi_3)^{T}$ is given by
\begin{equation}\label{eq:phi}
\begin{aligned}
	\varphi_1&=\frac{dt_1\wedge dt_2}{t_1t_2} \\
	\varphi_2&=-\frac{dt_1\wedge dt_2}{t_1(1-t_2)}-\frac{dt_1\wedge dt_2}{t_2(1-t_1)} \\
	\varphi_3&=\frac{dt_1\wedge dt_2}{(1-t_1)(1-t_2)}.
\end{aligned}
\end{equation}
For the cycles considered in Section \ref{sec:asymptoticanalysis}, one can verify directly that \eqref{eq:cDFint} gives solutions of \eqref{eq:cDFeq} by differentiation under the integral sign and integration by parts.
We denote the integral \eqref{eq:cDFint} by $V_{\Delta}(x)$ when we want to make clear the choice of the cycle $\Delta$. 
In the following, we mainly treat the system \eqref{eq:cDFeq} and its integral representation of solutions \eqref{eq:cDFint}.

\section{Asymptotic analysis at infinity}
\label{sec:asymptoticanalysis}
In this section, we examine the Stokes phenomenon of the confluent Dotsenko--Fateev system \eqref{eq:cDFeq} at infinity. 
Since $T$ has distinct eigenvalues, the system \eqref{eq:cDFeq}
admits a unique formal fundamental solution matrix of the form
\begin{equation}\label{eq:FSM}
  \widehat{\calY}(x)=\widehat{F}(x)x^\Lambda e^{Tx},
  \qquad
  \widehat{F}(x)=I+\sum_{n=1}^{\infty}F_nx^{-n},
\end{equation}
where
\[
  \Lambda=\operatorname{diag}(A)=\operatorname{diag}[2a+2b+g,\ a+g,\ -2b-g].
\]
The matrices $F_n$ are determined recursively by substituting
\eqref{eq:FSM} into \eqref{eq:cDFeq}.
Below, we construct sectorial solutions using the integral
representation \eqref{eq:cDFint} and show that they admit
asymptotic expansions given by constant multiples of the columns
of $\widehat{\calY}(x)$.

\subsection{Sectorial solutions and their asymptotics}
\label{sec:SecandAsy}
In what follows, we shall construct a fundamental solution matrix defined on a sector in terms of the integral \eqref{eq:cDFint}.
To justify the convergence of integrals and the asymptotic analysis below,
we assume for the moment that 
	\begin{equation}\label{eq:assumpconv}
	\re(-a-b-g),~\re(b),~\re(g)>0. 
	\end{equation}
These conditions ensure the absolute convergence of all the integrals appearing below.
The resulting formulas will later be extended meromorphically with respect to $(a,b,g)$. 

To describe the integration cycles, we prepare some notation following Haraoka \cite{Haraoka2024}. 
For $\alpha,\beta\in\bbC\cup\{\infty e^{\sqrt{-1}\theta}\mid\theta\in\bbR\}$ and $i=1,2$, we denote by $(t_i:\alpha\to\beta)$ the open segment in the complex $t_i$-plane from $\alpha$ to $\beta$. 
Then, for $\theta\in\bbR$, we define three domains $\Delta_{ij}^{\theta}$ $(ij\in\{00,10,11\})$ by
\begin{equation}
\begin{aligned}
\Delta_{00}^{\theta}&:=(t_1:0\to \infty e^{\sqrt{-1}\theta})\times (t_2:0\to t_1), \\
\Delta_{10}^{\theta}&:=(t_1:1\to \infty e^{\sqrt{-1}\theta})\times (t_2:0\to \infty e^{\sqrt{-1}\theta}), \\
\Delta_{11}^{\theta}&:=(t_1:1\to \infty e^{\sqrt{-1}\theta})\times (t_2:1\to t_1).
\end{aligned}
\end{equation}
For each domain, we set a branch of the integrand
	\begin{equation}\label{eq:integrand}
	\Phi(t_1,t_2;x)=(t_1-t_2)^{g}\prod_{l=1}^{2}t_l^{-(a+b+g)}(1-t_l)^{b}e^{xt_l}
	\end{equation}
as follows: 
\begin{align}
&\Delta_{00}^{\theta}:
	\begin{cases}
	\arg t_1=\theta, \\
	\arg(1-t_1)=0 \quad \text{at $t_1=0$}, \\
	\arg t_2=\theta, \\
	\arg(1-t_2)=0 \quad \text{at $t_2=0$}, \\
	\arg(t_1-t_2)=\theta,
	\end{cases}  \label{eq:branch1}
\\
&\Delta_{10}^{\theta}:
	\begin{cases}
	\arg t_1=0 \quad \text{at $t_1=1$}, \\
	\arg(1-t_1)=\theta+\pi, \\
	\arg t_2=\theta, \\
	\arg(1-t_2)=0 \quad \text{at $t_2=0$}, \\
	\arg(t_1-t_2)=0 \quad \text{at $(t_1,t_2)=(1,0)$},
	\end{cases}  \label{eq:branch2}
\\
&\Delta_{11}^{\theta}:
	\begin{cases}
	\arg t_1=0 \quad \text{at $t_1=1$}, \\
	\arg(1-t_1)=\theta+\pi, \\
	\arg t_2=0 \quad \text{at $t_2=1$}, \\
	\arg(1-t_2)=\theta+\pi , \\
	\arg(t_1-t_2)=\theta.
	\end{cases}  \label{eq:branch3}
\end{align}
We define the sectors $\Sigma_{k}$ $(k\in\bbZ)$ by
	\begin{equation}\label{eq:Sigma}
	\Sigma_{k}:=\left\{x\in \bbC^{\times}\,\middle| \,-\frac{3}{2}\pi+k\pi<\arg x<\frac{\pi}{2}+k\pi\right\}.
	\end{equation}
Here we regard these sectors as sectors on the Riemann surface of the logarithm of $\bbC^{\times}$.
For $x\in\Sigma_{k}$, by taking $\theta\in\bbR$ such that $\cos(\arg x+\theta)<0$, the integral \eqref{eq:cDFint} converges and gives a solution to \eqref{eq:cDFeq}. 
The following lemma can be shown easily. 
\begin{lemma}\label{lem:theta}
For any $x\in\Sigma_{k}$, there exists $\theta\in(-(k+1)\pi,-k\pi)$ such that $\cos(\arg x+\theta)<0$. 
\end{lemma}
For $x\in\Sigma_k$, choose $\theta\in (-(k+1)\pi,-k\pi)$ such that $\cos(\arg x+\theta)<0$. 
By Cauchy's theorem, the value of $V_{\Delta_{ij}^{\theta}}(x)$ is independent
of the admissible choice of such $\theta$.
Then we define the functions $V_{ij}^{(k)}(x)$ ($ij\in \{00,10,11\}$) on $\Sigma_{k}$ by 
	\begin{equation}\label{eq:Vijvec}
	V_{ij}^{(k)}(x):=V_{\Delta_{ij}^{\theta}}(x)
	=\int_{\Delta_{ij}^{\theta}}(t_1-t_2)^{g}\prod_{l=1}^{2}t_l^{-(a+b+g)}(1-t_l)^{b}e^{xt_l} \,\vec{\varphi}.
	\end{equation}
Here, when $x$ varies continuously within the sector, we choose $\theta\in (-(k+1)\pi,-k\pi)$ continuously so that $\cos(\arg x+\theta)<0$.
Then $V_{ij}^{(k)}(x)$ gives a solution to \eqref{eq:cDFeq} on $\Sigma_{k}$. 

\begin{proposition}\label{prop:firstterm}
Let $k\in\bbZ$ and $\Sigma'\Subset \Sigma_k$ be any proper closed subsector. 
Then it holds that 
	\begin{align}
	V_{00}^{(k)}(x)&=x^{2a+2b+g}(c_{00}\vec{e}_1+o(1)), \label{eq:leadingterm1}\\
	V_{10}^{(k)}(x)&=e^{x}x^{a+g}(c_{10}\vec{e}_2+o(1)),\label{eq:leadingterm2} \\
	V_{11}^{(k)}(x)&=e^{2x}x^{-2b-g}(c_{11}\vec{e}_3+o(1))\label{eq:leadingterm3}
	\end{align}
as $x\to\infty$, uniformly for $x\in\Sigma'$. 
Here $\vec{e}_1=(1,0,0)^T$, $\vec{e}_{2}=(0,1,0)^T$, $\vec{e}_{3}=(0,0,1)^T$ and
	\begin{align}
	c_{00}&=e^{\pi\sqrt{-1}(2a+2b+g)}\frac{\Gamma(-a-b-g)\Gamma(-a-b-\frac{g}{2})\Gamma(g)}{\Gamma(\frac{g}{2})}, \label{eq:c00}\\
	c_{10}&=e^{\pi\sqrt{-1}(a+b+g)}\Gamma(b)\Gamma(-a-b-g), \\
	c_{11}&=e^{-\pi\sqrt{-1}g}\frac{\Gamma(b)\Gamma(b+\frac{g}{2})\Gamma(g)}{\Gamma(\frac{g}{2})}.
\end{align}
\end{proposition}
\begin{proof}
We fix $k\in\mathbb{Z}$ and a proper closed subsector $\Sigma' \Subset \Sigma_{k}$. 
It is enough to assume that the opening of $\Sigma'$ is smaller than $\pi$: any proper closed subsector in $\Sigma_k$ can be covered by a finite union of such subsectors. 
Then we can choose a constant $\delta\in(-\pi/2,\pi/2)$ such that, for each $x\in\Sigma'$, there exists an admissible $\theta=\theta(x)$ satisfying
\[
  \delta=\arg x+\theta+\pi.
\]
Moreover, we choose $\delta$ so that the corresponding $\theta$'s range over a compact subinterval of $(-(k+1)\pi,-k\pi)$.
For this admissible choice of $\theta$, we write $\Delta_{ij}:=\Delta_{ij}^{\theta}$.
To simplify notation, we suppress the superscript
$(k)$ and write $V_{ij}(x):=V_{ij}^{(k)}(x)$ throughout this proof. 
Let us introduce the scalar integral
	\begin{equation}
	\begin{aligned}
	I_{ij}(a,b,g;x):=\int_{\Delta_{ij}}\Phi(t_1,t_2;x)\varphi_1=\int_{\Delta_{ij}}(t_1-t_2)^{g}\prod_{l=1}^{2}t_l^{-(a+b+g)}(1-t_l)^{b}e^{xt_l}\frac{dt_1\wedge dt_2}{t_1 t_2}.
	\end{aligned}
	\end{equation}
Then, noting that 
	\begin{equation}
	-\frac{1}{t_1(1-t_2)}-\frac{1}{t_2(1-t_1)}=\frac{1}{t_1t_2}+\frac{1}{(1-t_1)(1-t_2)}-\frac{1}{t_1t_2(1-t_1)(1-t_2)},
	\end{equation}
we have
	\begin{equation}
	\int_{\Delta_{ij}}\Phi(t_1,t_2;x)\varphi_2=I_{ij}(a,b,g;x)+I_{ij}(a,b-1,g;x)-I_{ij}(a+1,b-1,g;x).
	\end{equation}
Besides, we have
	\begin{equation}
	\int_{\Delta_{ij}}\Phi(t_1,t_2;x)\varphi_3=I_{ij}(a,b-1,g;x).
	\end{equation}
Consequently, we obtain
	\begin{equation}\label{eq:Vij_Iij}
	V_{ij}(x)=\begin{pmatrix}
	I_{ij}(a,b,g;x) \\
	I_{ij}(a,b,g;x)+I_{ij}(a,b-1,g;x)-I_{ij}(a+1,b-1,g;x) \\
	I_{ij}(a,b-1,g;x)
	\end{pmatrix}.
	\end{equation}
Then, we shall examine the asymptotic behavior of $I_{ij}(a,b,g;x)$. 
We first consider the integral $I_{00}(a,b,g;x)$. 
Put
	\begin{equation}
	u_1=e^{\pi\sqrt{-1}}xt_1, 
	\quad 
	u_2=e^{\pi\sqrt{-1}}xt_2.
	\end{equation}
Then the domain $\Delta_{00}$ changes to $\widetilde{\Delta}_{00}:=(u_1:0\to\infty e^{\sqrt{-1}\delta})\times (u_2:0\to u_1)$ and we have
	\begin{equation}
	\begin{aligned}
	I_{00}(a,b,g;x)=&e^{\pi\sqrt{-1}(2a+2b+g)}x^{2a+2b+g} \\
	 &\times \int_{\widetilde{\Delta}_{00}}(u_1-u_2)^g \prod_{l=1}^{2}u_l^{-(a+b+g)}\left(1+\frac{u_l}{x}\right)^{b}e^{-u_l}\,\frac{du_1\wedge du_2}{u_1 u_2},
	\end{aligned}
	\end{equation}
where the branches of the integrand are specified as follows. 
	\begin{equation}
	\widetilde{\Delta}_{00}:
	\begin{cases}
	\arg u_1=\delta, \\
	\arg(1+u_1/x)=0 \quad \text{at $u_1=0$}, \\
	\arg u_2=\delta, \\
	\arg(1+u_2/x)=0 \quad \text{at $u_2=0$}, \\
	\arg(u_1-u_2)=\delta.
	\end{cases}  \label{eq:branch_tildedD11}
	\end{equation}
We consider taking the limit $x\to\infty$, $x\in \Sigma'$ for the integral. 
Interchanging the order of limit and integral (which will be justified later), we have 
	\begin{equation}\label{eq:integral00}
	\begin{aligned}
	\lim_{\substack{x\to\infty \\ x\in \Sigma'}}&\int_{\widetilde{\Delta}_{00}}(u_1-u_2)^g \prod_{l=1}^{2}u_l^{-(a+b+g)}\left(1+\frac{u_l}{x}\right)^{b}e^{-u_l}\,\frac{du_1\wedge du_2}{u_1 u_2} \\
	&=\int_{\widetilde{\Delta}_{00}}(u_1-u_2)^g \prod_{l=1}^{2}u_l^{-(a+b+g+1)}e^{-u_l}\,du_1\wedge du_2. 
	\end{aligned}
	\end{equation}
By Cauchy's theorem, we may rotate both variables in the limiting
integral simultaneously from argument $\delta$ to $0$, using
$|\delta|<\pi/2$ and the convergence assumptions.
Thus the limiting integral equals the integral over
$0<u_2<u_1<\infty$.

Here we recall the Laguerre--Selberg integral formula (e.g. the case of $(n,k)=(2,0)$ of Corollary 8.2.2 in Andrews--Askey--Roy \cite{AARBook}).
	For $\alpha,\gamma\in\bbC$ such that $\re(\alpha)>0$ and $\re(\gamma)>-\min\{1/2,\re(\alpha)\}$, it holds that
	\begin{equation}
	\iint_{(0,\infty)^2}|u_1-u_2|^{2\gamma}\prod_{l=1}^{2}u_l^{\alpha-1}e^{-u_l}\,du_1du_2=\frac{\Gamma(\alpha)\Gamma(\alpha+\gamma)\Gamma(1+2\gamma)}{\Gamma(1+\gamma)}.
	\end{equation}
This implies
	\begin{equation}\label{eq:LSformula}
	\iint_{0<u_2<u_1<\infty}(u_1-u_2)^{2\gamma}\prod_{l=1}^{2}u_l^{\alpha-1}e^{-u_l}\,du_1du_2=\frac{1}{2}\frac{\Gamma(\alpha)\Gamma(\alpha+\gamma)\Gamma(1+2\gamma)}{\Gamma(1+\gamma)}.
	\end{equation}
Applying this formula with $\alpha=-a-b-g$ and $\gamma=g/2$, we get
	\begin{equation}
\iint_{\widetilde{\Delta}_{00}}(u_1-u_2)^{g}\prod_{l=1}^{2}u_l^{-(a+b+g+1)}e^{-u_l} du_1du_2
	=\frac{1}{2}\frac{\Gamma(-a-b-g)\Gamma(-a-b-\frac{g}{2})\Gamma(1+g)}{\Gamma(1+\frac{g}{2})}. 
	\end{equation}
By using
	\begin{equation}
	\frac{\Gamma(1+g)}{\Gamma(1+\frac{g}{2})}=\frac{g\Gamma(g)}{\frac{g}{2}\Gamma(\frac{g}{2})}=2\frac{\Gamma(g)}{\Gamma(\frac{g}{2})},
	\end{equation}
we obtain 
	\begin{equation}
	I_{00}(a,b,g;x)=x^{2a+2b+g}(c_{00}+o(1)),
	\end{equation}
where $c_{00}$ is given by 	\eqref{eq:c00}. 
Since $c_{00}$ depends only on $a+b$ and $g$, it is unchanged by the shift
$(a,b)\mapsto(a+1,b-1)$. The corresponding leading terms in the second
component of \eqref{eq:Vij_Iij} therefore cancel, while
$I_{00}(a,b-1,g;x)=x^{2a+2b+g}O(x^{-2})$.
Thus \eqref{eq:leadingterm1} follows.

We next consider the integral $I_{10}(a,b,g;x)$.
We fix the branch $\arg(t_1-1)=\theta$ on $\Delta_{10}$ and put
\[
  u_1=e^{\pi\sqrt{-1}}x(t_1-1),\qquad u_2=e^{\pi\sqrt{-1}}xt_2.
\]
Then the domain $\Delta_{10}$ changes to $
  \widetilde{\Delta}_{10}
  =(u_1:0\to\infty e^{\sqrt{-1}\delta})
   \times(u_2:0\to\infty e^{\sqrt{-1}\delta})$,
and we have
\begin{equation}\label{eq:Int10}
\begin{aligned}
&I_{10}(a,b,g;x)
   =e^{\pi\sqrt{-1}(a+b+g-1)}e^{x}x^{a+g-1} \\
  &\times\int_{\widetilde{\Delta}_{10}}
  \left(1-\frac{u_1-u_2}{x}\right)^g
  \left(1-\frac{u_1}{x}\right)^{-(a+b+g+1)}
  \left(1+\frac{u_2}{x}\right)^b 
  u_1^b u_2^{-(a+b+g+1)}e^{-u_1-u_2}
  \,du_1\wedge du_2.
\end{aligned}
\end{equation}
Here $\arg u_1=\arg u_2=\delta$, and the branches of the three factors
containing $x$ are determined by continuous logarithms that vanish at $(u_1,u_2)=(0,0)$. 

Interchanging the limit and the integral, as justified below, and using
the definition of the gamma function along the ray of argument $\delta$, we obtain
\begin{equation}\label{eq:I10leading}
I_{10}(a,b,g;x)
  =e^{x}x^{a+g-1}
  \left(e^{\pi\sqrt{-1}(a+b+g-1)}\Gamma(b+1)\Gamma(-a-b-g)+o(1)\right).
\end{equation}
Applying this formula with parameters $(a,b-1,g)$ and $(a+1,b-1,g)$ gives
\[
\begin{aligned}
I_{10}(a,b-1,g;x)&=e^{x}x^{a+g}O(x^{-1}),\\
I_{10}(a+1,b-1,g;x)
  &=e^{x}x^{a+g}
    \left(e^{\pi\sqrt{-1}(a+b+g-1)}\Gamma(b)\Gamma(-a-b-g)+o(1)\right)\\
  &=e^{x}x^{a+g}(-c_{10}+o(1)).
\end{aligned}
\]
Together with \eqref{eq:Vij_Iij} and \eqref{eq:I10leading}, these relations
give \eqref{eq:leadingterm2}.

Finally, we consider $I_{11}(a,b,g;x)$ and put
\[
  u_1=e^{\pi\sqrt{-1}}x(t_1-1),\qquad u_2=e^{\pi\sqrt{-1}}x(t_2-1).
\]
Then the domain $\Delta_{11}$ changes to
  $\widetilde{\Delta}_{11}
  =(u_1:0\to\infty e^{\sqrt{-1}\delta})\times(u_2:0\to u_1)
  $,
and we have
\begin{equation}
\begin{aligned}
I_{11}(a,b,g;x)
  &=e^{-\pi\sqrt{-1}g}e^{2x}x^{-2b-g-2}\\
  &\quad\times\int_{\widetilde{\Delta}_{11}}(u_1-u_2)^g
  \prod_{l=1}^{2}
  u_l^b\left(1-\frac{u_l}{x}\right)^{-(a+b+g+1)}e^{-u_l}
  \,du_1\wedge du_2.
\end{aligned}
\end{equation}
Here $\arg u_1=\arg u_2=\arg(u_1-u_2)=\delta$, and
$\arg(1-u_l/x)=0$ at $u_l=0$.
Taking the limit and rotating the common ray to the positive real axis,
we apply \eqref{eq:LSformula} with $\alpha=b+1$ and $\gamma=g/2$ to obtain
\begin{equation}\label{eq:I11leading}
\begin{aligned}
I_{11}(a,b,g;x)
  =e^{2x}x^{-2b-g-2}
  \left(
  e^{-\pi\sqrt{-1}g}
  \frac{\Gamma(b+1)\Gamma(b+1+\frac{g}{2})\Gamma(g)}
       {\Gamma(\frac{g}{2})}
  +o(1)\right).
\end{aligned}
\end{equation}
The leading coefficient is independent of $a$. 
Therefore
\[
\begin{aligned}
I_{11}(a,b-1,g;x)&=e^{2x}x^{-2b-g}(c_{11}+o(1)),\\
I_{11}(a+1,b-1,g;x)&=e^{2x}x^{-2b-g}(c_{11}+o(1)).
\end{aligned}
\]
The leading terms in the second component of \eqref{eq:Vij_Iij} cancel,
while \eqref{eq:I11leading} gives
$I_{11}(a,b,g;x)=e^{2x}x^{-2b-g}O(x^{-2})$.
This proves \eqref{eq:leadingterm3}.

It remains to justify the interchanges of limit and integral and the
uniformity of the resulting limits for $x\in\Sigma'$.
We give the details for $I_{10}(a,b,g;x)$, since the other cases are treated in a similar way. 
For the integral
	\begin{equation}\label{eq:Int10part}
	\int_{\widetilde{\Delta}_{10}}
  \left(1-\frac{u_1-u_2}{x}\right)^g
  \left(1-\frac{u_1}{x}\right)^{-(a+b+g+1)}
  \left(1+\frac{u_2}{x}\right)^b
  u_1^b u_2^{-(a+b+g+1)}e^{-u_1-u_2}
  \,du_1\wedge du_2
	\end{equation}
on the right-hand side of \eqref{eq:Int10}, we set
	\begin{equation}\label{eq:ulrl}
	u_l=r_le^{\sqrt{-1}\delta}\quad (r_l>0,~l=1,2).
	\end{equation}
By our choice of $\delta$, the values of $\theta(x)=\delta-\arg x-\pi$
lie in a fixed compact subinterval of $(-(k+1)\pi,-k\pi)$.
Therefore, with the branches fixed above, there exist constants $C_1,N>0$ such that 
\begin{equation}
\begin{aligned}
  &\left|
  \left(1-\frac{u_1-u_2}{x}\right)^g
  \left(1-\frac{u_1}{x}\right)^{-(a+b+g+1)}
  \left(1+\frac{u_2}{x}\right)^b
  \right|
  \leq C_1(1+r_1+r_2)^N
\end{aligned}
\end{equation}
uniformly for $x\in\Sigma'$ and $|x|\geq1$. 
Moreover, we have
	\begin{equation}
 	\left|u_1^b u_2^{-(a+b+g+1)}e^{-u_1-u_2}\right|\le C_2 r_1^{\re(b)}r_2^{\re(-a-b-g)-1}e^{-(r_1+r_2)\cos\delta}
	\end{equation}
for some $C_2>0$. 
Consequently, after the parametrization \eqref{eq:ulrl}, the absolute value of the integrand of \eqref{eq:Int10part} is bounded by
	\begin{equation}
	C_1C_2(1+r_1+r_2)^N r_1^{\re(b)}r_2^{\re(-a-b-g)-1}e^{-(r_1+r_2)\cos\delta}
	\le C r_1^{\re(b)}r_2^{-\re(a+b+g)-1}
    e^{-h(r_1+r_2)\cos\delta}
	\end{equation}
for some $C>0$ and $0<h<1$.  
Here we have used $\cos\delta>0$.
The last expression is independent of $x$ and is integrable on
	$(0,\infty)^2$, since
	$\re(b)>-1$ and $\re(-a-b-g)>0$
under the assumption \eqref{eq:assumpconv}. 
Moreover, for each fixed $r_1,r_2>0$, the factors depending on $x$ converge to $1$ uniformly for $x\in\Sigma'$ as $|x|\to\infty$.
Hence, by the dominated convergence theorem, the integral converges
to its limiting integral uniformly for $x\in\Sigma'$ as $|x|\to\infty$.
In particular, the interchange of limit and integral is justified.

The same argument applies to the shifted parameters
$(a,b-1,g)$ and $(a+1,b-1,g)$ occurring in \eqref{eq:Vij_Iij}, since the corresponding integrals also converge under \eqref{eq:assumpconv}.
\end{proof}
From this proposition, we see that $V^{(k)}_{00}(x),V^{(k)}_{10}(x)$ and $V^{(k)}_{11}(x)$ are linearly independent on $\Sigma_{k}$.
Therefore, 
	\begin{equation}\label{eq:calV}
	\calV^{(k)}(x):=(V_{00}^{(k)}(x),V_{10}^{(k)}(x),V_{11}^{(k)}(x))
	\end{equation}
gives a fundamental solution matrix of \eqref{eq:cDFeq} on $\Sigma_{k}$ and satisfies
	\begin{equation}
	\calV^{(k)}(x)=\left(C+o(1)\right)x^{\Lambda}e^{Tx},
	\quad C=\operatorname{diag}[c_{00},c_{10},c_{11}]
	\end{equation}
as $x\to\infty$, uniformly for $x\in\Sigma'$. 

We next show that $\calV^{(k)}(x)$ admits an asymptotic expansion on $\Sigma_{k}$. 
We set 
	\begin{equation}\label{eq:calY}
	\calY^{(k)}(x)=\calV^{(k)}(x)C^{-1}.
	\end{equation}
Then $\calY^{(k)}(x)$ gives a fundamental solution matrix of \eqref{eq:cDFeq} and satisfies 
	\begin{equation}
	\calY^{(k)}(x)=\left(I+o(1)\right)x^{\Lambda}e^{Tx}
	\end{equation}
since $x^{\Lambda}$ and $e^{Tx}$ are both diagonal.

\begin{theorem}\label{thm:asymptoticexpansion}
Let $k\in\bbZ$.
The fundamental solution matrix $\calY^{(k)}(x)$ defined by
\eqref{eq:calY} admits the asymptotic expansion
$\widehat{\calY}(x)$ given by \eqref{eq:FSM} as $x\to\infty$
in $\Sigma_k$.
More precisely, for every proper closed subsector
$\Sigma'\Subset\Sigma_k$ and every $N\geq1$,
\[
  \calY^{(k)}(x)e^{-Tx}x^{-\Lambda}
  -\sum_{n=0}^{N-1}F_nx^{-n}
  =O(x^{-N}),
  \qquad F_0=I,
\]
uniformly for $x\in\Sigma'$.
Consequently, $\calV^{(k)}(x)$ admits the asymptotic expansion
\[
  \widehat{\calV}(x):=\widehat{\calY}(x)C.
\]
\end{theorem}

\begin{proof}
The system \eqref{eq:cDFeq} has Poincar\'e rank one at infinity,
and $T$ has distinct real eigenvalues.
Hence the sectors specified in
Corollary~4.8 and equation~(4.8) of \cite{CottiGuzzettiMasoero}
coincide with $\Sigma_k$ after reindexing.
By this result, there exists a fundamental solution matrix
$\widetilde{\calY}^{(k)}(x)$ near infinity in $\Sigma_k$
admitting the asymptotic expansion $\widehat{\calY}(x)$.
It extends to $\Sigma_k$ by analytic continuation.

Since $\calY^{(k)}(x)$ and $\widetilde{\calY}^{(k)}(x)$ solve the same
system, there is a constant matrix $R$ such that
\[
  \calY^{(k)}(x)=\widetilde{\calY}^{(k)}(x)R.
\]
Comparing their leading terms using
Proposition~\ref{prop:firstterm}, we obtain
\[
  x^\Lambda e^{Tx}R e^{-Tx}x^{-\Lambda}=I+o(1)
\]
as $x\to\infty$, uniformly on every proper closed subsector
of $\Sigma_k$.
On the other hand, the $(i,j)$ entry of the left-hand side is
\[
  R_{ij}x^{\Lambda_{ii}-\Lambda_{jj}}
  e^{(T_{ii}-T_{jj})x}.
\]
For $i\ne j$, the sector $\Sigma_k$ contains a ray along which
the exponential factor grows exponentially.
Thus the above convergence forces $R_{ij}=0$.
The diagonal entries give $R_{ii}=1$, so $R=I$.
Therefore $\calY^{(k)}(x)=\widetilde{\calY}^{(k)}(x)$, proving
the asserted asymptotic expansion.
The statement for $\calV^{(k)}(x)=\calY^{(k)}(x)C$ follows immediately.
\end{proof}
\begin{remark}\label{rem:uniqueness}
The argument in the proof of Theorem~\ref{thm:asymptoticexpansion}
also shows that $\calV^{(k)}(x)$ is the unique fundamental solution
matrix on $\Sigma_k$ admitting the asymptotic expansion
$\widehat{\calV}(x)$.
\end{remark}
Since the coefficient matrix of \eqref{eq:cDFeq} depends
holomorphically on $(a,b,g)$ and $T$ is fixed with pairwise
distinct eigenvalues, standard asymptotic theory with parameters implies that the normalized sectorial fundamental solution matrices $\widetilde{\calY}^{(k)}(x)$ in the proof of Theorem~\ref{thm:asymptoticexpansion} exist for all $(a,b,g)\in\mathbb{C}^3$ and depend holomorphically on the parameters.
Under \eqref{eq:assumpconv}, they coincide with $\calY^{(k)}(x)$,
so $\widetilde{\calY}^{(k)}(x)C$ gives a meromorphic continuation
of $\calV^{(k)}(x)$. 

By the gamma duplication formula
	\[
	\frac{\Gamma(g)}{\Gamma(g/2)}
	=\frac{2^{g-1}}{\sqrt{\pi}}\,
	\Gamma\!\left(\frac{g+1}{2}\right),
	\]
we have $C\in\mathrm{GL}(3,\bbC)$ if and only if
	\begin{equation}\label{eq:assumpStokes}
	-a-b-g,\quad -a-b-\frac{g}{2},\quad
	b,\quad b+\frac{g}{2},\quad \frac{g+1}{2}
	\notin\mathbb{Z}_{\leq 0}.
	\end{equation}
Under this condition, the continued matrix $\widetilde{\calY}^{(k)}(x)C$
is a fundamental solution matrix on $\Sigma_k$, and the conclusions
of Proposition~\ref{prop:firstterm},
Theorem~\ref{thm:asymptoticexpansion} and
Remark~\ref{rem:uniqueness} remain valid for the continued solutions.
For the rest of this section, we assume \eqref{eq:assumpStokes}
and use $\calV^{(k)}(x)$ to denote the continued fundamental
solution matrix.

\subsection{Stokes matrices and monodromy}
Now we define the Stokes matrices. 
Set 
	\begin{equation}
	\Pi_{k}:=\Sigma_{k}\cap\Sigma_{k+1}=\left\{x\in \bbC^{\times}\,\middle| \,-\frac{\pi}{2}+k\pi<\arg x<\frac{\pi}{2}+k\pi\right\}.
	\end{equation}
On the sector $\Pi_{k}$, there are two fundamental solution matrices $\calV^{(k)}(x)$ and $\calV^{(k+1)}(x)$. 
Thus, there exists a unique invertible matrix $S_{k}\in\mathrm{GL}(3,\bbC)$ such that
	\begin{equation}\label{eq:DefStokes}
	\calV^{(k+1)}(x)=\calV^{(k)}(x)S_{k}
	\quad 
	\text{on~$\Pi_{k}$}.
	\end{equation}
\begin{definition}
We call the matrices $\{S_{k}\}_{k\in\bbZ}$ the Stokes matrices associated with the sectorial fundamental solution matrices $\{\calV^{(k)}(x)\}_{k\in\bbZ}$.
\end{definition}
Combining Theorem~\ref{thm:asymptoticexpansion} and the uniqueness of the fundamental solution matrix admitting the prescribed asymptotic expansion (cf. Remark~\ref{rem:uniqueness}), we have the following.
\begin{lemma}\label{lem:Stokespreiod}
	\begin{equation}\label{eq:Stokespreiod}
	e^{2\pi\sqrt{-1}\Lambda}S_{k+2}=S_{k}e^{2\pi\sqrt{-1}\Lambda},
	\quad k\in\bbZ.
	\end{equation}
\end{lemma}
\begin{proof}
From Theorem~\ref{thm:asymptoticexpansion} and the uniqueness in Remark~\ref{rem:uniqueness}, it follows that
	\begin{equation}\label{eq:kperiod}
	\calV^{(k+2)}(x)=\calV^{(k)}(xe^{-2\pi\sqrt{-1}})e^{2\pi\sqrt{-1}\Lambda}, 
	\quad x\in\Sigma_{k+2} 
	\end{equation}
for any $k\in\bbZ$. 
Indeed, the right-hand side has the same asymptotic expansion $\widehat{\calV}(x)$ as the left-hand side on $\Sigma_{k+2}$.
Now we take $x\in\Pi_{k+2}$.  
Then, noting $\Pi_{k+2}=\Sigma_{k+2}\cap\Sigma_{k+3}$ and $xe^{-2\pi\sqrt{-1}}\in\Pi_{k}$, we have
	\begin{equation}
	\calV^{(k+3)}(x)=\calV^{(k+1)}(xe^{-2\pi\sqrt{-1}})e^{2\pi\sqrt{-1}\Lambda}=\calV^{(k)}(xe^{-2\pi\sqrt{-1}})S_{k}e^{2\pi\sqrt{-1}\Lambda}.
	\end{equation}
On the other hand, by the definition of $S_{k+2}$ and the above identity,
	\begin{equation}
	\calV^{(k+3)}(x)=\calV^{(k+2)}(x)S_{k+2}=\calV^{(k)}(xe^{-2\pi\sqrt{-1}})e^{2\pi\sqrt{-1}\Lambda}S_{k+2}. 
	\end{equation}
Since $\calV^{(k)}(xe^{-2\pi\sqrt{-1}})$ is invertible, we obtain $S_k e^{2\pi\sqrt{-1}\Lambda}=e^{2\pi\sqrt{-1}\Lambda}S_{k+2}$. 
\end{proof}
Thanks to this lemma, the derivation of the full collection of the Stokes matrices $\{S_{k}\}_{k\in\bbZ}$ is reduced to the derivation of two matrices $S_{-1}$ and $S_{0}$. 
\begin{theorem}\label{thm:StokesMat}
	Assume \eqref{eq:assumpStokes}.
	For each $k\in\bbZ$, let $\Sigma_{k}$ be the sector defined by \eqref{eq:Sigma} and $\calV^{(k)}(x)$ be the fundamental solution matrix on $\Sigma_{k}$, defined by the meromorphic continuation of \eqref{eq:calV} in $(a,b,g)$. 
	Then, the Stokes matrices $\{S_{k}\}_{k\in\bbZ}$ associated with $\{\calV^{(k)}(x)\}_{k\in\bbZ}$ are determined by 
	\begingroup\small
	\begin{align}
	&S_{-1}=\begin{pmatrix}
	1 & 0 & 0 \\
	1-e(-2b) & 1 & 0 \\
	\left(1-e(-2b))(1-e(-2b-g)\right) & \left(1-e(-2b-g)\right)\left(1+e(g)\right) & 1
	\end{pmatrix}, \label{eq:MainStokes1}
	\\
	&S_{0}=\begin{pmatrix}
	 1 & -\left(1-e(-2a-2b-g)\right)\left(1+e(g)\right) &
  e(-g) \left(1-e(-2a-2b-g)\right)\left(1-e(-2a-2b-2g)\right) \\
 	0 & 1 & -e(-g) \left(1-e(-2a-2b-2g)\right) \\
 	0 & 0 & 1 
	\end{pmatrix}\label{eq:MainStokes}
	\end{align}
\endgroup
together with the recurrence relation \eqref{eq:Stokespreiod}.
Here $e(\alpha)=e^{\pi \sqrt{-1} \alpha}$. 
\end{theorem}
We remark that the Stokes matrices associated with the normalized fundamental solution matrices $\{\calY^{(k)}(x)\}_{k\in\bbZ}$ defined by \eqref{eq:calY} are given by $\{CS_kC^{-1}\}_{k\in\bbZ}$.
The proof of this theorem will be given in the next subsection.
Before that, we give an explicit description of the monodromy representation of \eqref{eq:cDFeq} associated with the fundamental solution matrix $\calV^{(0)}(x)$.
Set $X:=\bbP^{1}\setminus\{0,\infty\}$.
We fix a base point $x_0\in X$ and choose the branch of $\arg x_0$ so that $x_0$ is regarded as a point of $\Sigma_{0}$. 
Denote by $\gamma_0$ the homotopy class in $\pi_1(X,x_0)$ of a positively oriented simple loop based at $x_0$ encircling $x=0$. 
We set $\gamma_{\infty}:=\gamma_{0}^{-1}$, which gives the homotopy class of a positively oriented simple loop based at $x_0$ encircling $x=\infty$. 
We denote by $\gamma\left[\calV^{(0)}(x)\right]$ the analytic continuation of $\calV^{(0)}(x)$ along the loop $\gamma\in\pi_{1}(X,x_0)$. 
Then, we obtain
	\begin{equation}
	\gamma_{\infty}\left[\calV^{(0)}(x)\right]=\calV^{(0)}(x)S_0 e^{-2\pi\sqrt{-1}\Lambda}S_{-1} 
	\end{equation}
from the definition of the Stokes matrices \eqref{eq:DefStokes} and the recurrence relation \eqref{eq:Stokespreiod}. 
Hence, by setting $M_{\infty}=S_0 e^{-2\pi\sqrt{-1}\Lambda}S_{-1}$ and $M_{0}=M_{\infty}^{-1}$, we have the monodromy representation
	\begin{equation}
	\begin{array}{cccc}
	\rho: & \pi_1(X,x_0) & \to & \mathrm{GL}(3,\bbC) \\
	& \gamma_{0} & \mapsto & M_0 \\ 
	& \gamma_{\infty} & \mapsto & M_\infty \\
	\end{array}.
	\end{equation}
\begin{remark}
To obtain the complete monodromy data, we have to derive the \emph{central connection matrices} $\{C_{k}\}_{k\in\bbZ}$ such that
	\begin{equation}
	\calV^{(k)}(x)=\calV_{0}(x)C_{k} \quad \text{on~$\Sigma_{k}\cap \{x\in\mathbb{C}\mid |x|<r\}$}
	\end{equation}
for some $r>0$. 
Here $\calV_{0}(x)$ denotes a fundamental solution matrix around $x=0$: 
if $a,a+g,2a+g\notin\mathbb{Z}$, it is given by
	\begin{equation}
	\calV_{0}(x)=G(x)x^{D},\quad G(x)=G_0+\sum_{m=1}^{\infty}G_{m}x^m,
	\end{equation} 
where $D:=\operatorname{diag}[0,a,2a+g]$ and $G_0\in\mathrm{GL}(3,\bbC)$ satisfy $G_0^{-1}AG_0=D$. 
This problem will be discussed in a subsequent work.
\end{remark}

\subsection{Derivation of the Stokes matrices}

Let us prove Theorem~\ref{thm:StokesMat}. 
We first show the statement under the assumption \eqref{eq:assumpconv}, using the integral representations.
We follow Haraoka's contour-deformation method for the confluent $F_4$ system \cite{Haraoka2024}, adapting it to the integral representation of the confluent Dotsenko--Fateev system. 
More precisely, the integration cycles and their deformations closely parallel those used in Haraoka's work.

We derive only $S_{-1}$ since $S_0$ can be derived in the same way. 
To derive $S_{-1}$, we shall consider expressing $V_{ij}^{(0)}(x)$ ($ij\in\{00,10,11\}$) as a linear combination of $V_{00}^{(-1)}(x)$, $V_{10}^{(-1)}(x)$ and $V_{11}^{(-1)}(x)$ on $\Pi_{-1}$. 
For $x\in\Pi_{-1}$, let $\theta_{-1}\in(0,\pi)$ and $\theta_{0}\in(-\pi,0)$ be angles corresponding to the cycles of $V_{ij}^{(-1)}(x)$ and $V_{ij}^{(0)}(x)$ on $\Pi_{-1}$, respectively (e.g. Figure \ref{fig:Delta10}). 
\begin{figure}[!htbp]
\[
\left(\!\!\!
\vcenter{\hbox{%
\tikzset{every picture/.style={line width=0.75pt}} 

\begin{tikzpicture}[x=0.75pt,y=0.75pt,yscale=0.8,xscale=0.8]
\path[use as bounding box] (245,-30) rectangle (420,-120);

\filldraw[fill=white,draw=black] (296.5,-60) circle (2pt);
\filldraw[fill=white,draw=black] (352,-60) circle (2pt);

\draw[
  decoration={markings, mark=at position 0.5 with {\arrow[scale=1.2]{stealth}}},
  postaction={decorate}
] (352.5,-62.5) -- (405.5,-115);

\draw[dash pattern={on 0.84pt off 2.51pt}]
  (410,-60.5) -- (356,-60.5);

\draw (268.5,-38) -- (268.5,-59) -- (249,-59);

\draw[line width=0.5pt] (340,-59) ++(-3.6:25pt)
  arc[start angle=0, end angle=-36, radius=22pt];

\draw (290.5,-38) node[anchor=north west,inner sep=0.75pt] {$0$};
\draw (345.5,-38) node[anchor=north west,inner sep=0.75pt] {$1$};
\draw (250,-38) node[anchor=north west,inner sep=0.75pt] {$t_{1}$};
\draw (377,-65) node[anchor=north west,inner sep=0.75pt] {$\theta_{0}$};
\end{tikzpicture}
}}
\right)
\times
\left(\!\!\!
\vcenter{\hbox{%
\tikzset{every picture/.style={line width=0.75pt}} 

\begin{tikzpicture}[x=0.75pt,y=0.75pt,yscale=0.8,xscale=0.8]
\path[use as bounding box] (245,-30) rectangle (400,-120);

\filldraw[fill=white,draw=black] (296.5,-60) circle (2pt);
\filldraw[fill=white,draw=black] (352,-60) circle (2pt);

\draw[
  decoration={markings, mark=at position 0.5 with {\arrow[scale=1.2]{stealth}}},
  postaction={decorate}
] (297,-62.5) -- (350,-115);

\draw[dash pattern={on 0.84pt off 2.51pt}]
  (325,-60.5) -- (298,-60.5);

\draw (268.5,-38) -- (268.5,-59) -- (249,-59);

\draw[line width=0.5pt] (340-352+296.5,-59) ++(-3.6:25pt)
  arc[start angle=0, end angle=-36, radius=22pt];

\draw (290.5,-38) node[anchor=north west,inner sep=0.75pt] {$0$};
\draw (345.5,-38) node[anchor=north west,inner sep=0.75pt] {$1$};
\draw (250,-38) node[anchor=north west,inner sep=0.75pt] {$t_{2}$};
\draw (377-352+296.5,-65)
  node[anchor=north west,inner sep=0.75pt] {$\theta_{0}$};
\end{tikzpicture}
}}
\right)
\]
\caption{$\Delta_{10}^{\theta_0}$}
\label{fig:Delta10}
\end{figure}

Our strategy is as follows.
We continuously rotate the angle from $\theta_0$ to $\theta_{-1}$, which induces the deformation of domains $\Delta_{ij}^{\theta_{0}}$.
The resulting deformed domain can be expressed in terms of $\Delta_{ij}^{\theta_{-1}}$;  this expression gives linear combinations of $V_{ij}^{(-1)}(x)$ with coefficients arising from the differences between the branches of the integrand $\Phi(t_1,t_2;x)$ on $\Delta_{ij}^{\theta_{-1}}$ and those on the deformed domains.

On the domain $\Delta_{11}^{\theta_{0}}$, the resulting deformed domain is the same as $\Delta_{11}^{\theta_{-1}}$ (Figure \ref{fig:deform11}). 
\begin{figure}[!htbp]
\[
\begin{aligned}
&\left(\!\!\!
\vcenter{\hbox{%
\tikzset{every picture/.style={line width=0.75pt}} 

\begin{tikzpicture}[x=0.75pt,y=0.75pt,yscale=0.8,xscale=0.8]
\path[use as bounding box] (245,-30) rectangle (420,-120);

\filldraw[fill=white,draw=black] (296.5,-60) circle (2pt);
\filldraw[fill=white,draw=black] (352,-60) circle (2pt);

\draw[
  decoration={markings, mark=at position 0.5 with {\arrow[scale=1.2]{stealth}}},
  postaction={decorate}
] (352.5,-62.5) -- (405.5,-115);

\draw[dash pattern={on 0.84pt off 2.51pt}]
  (410,-60.5) -- (356,-60.5);

\draw (268.5,-38) -- (268.5,-59) -- (249,-59);

\draw[line width=0.5pt] (340,-59) ++(-3.6:25pt)
  arc[start angle=0, end angle=-36, radius=22pt];

\draw (290.5,-38) node[anchor=north west,inner sep=0.75pt] {$0$};
\draw (345.5,-38) node[anchor=north west,inner sep=0.75pt] {$1$};
\draw (250,-38) node[anchor=north west,inner sep=0.75pt] {$t_{1}$};
\draw (377,-65) node[anchor=north west,inner sep=0.75pt] {$\theta_{0}$};
\end{tikzpicture}
}}
\right)
\times
\left(\!\!\!
\vcenter{\hbox{%
\tikzset{every picture/.style={line width=0.75pt}} 

\begin{tikzpicture}[x=0.75pt,y=0.75pt,yscale=0.8,xscale=0.8]
\path[use as bounding box] (245,-30) rectangle (400,-120);

\filldraw[fill=white,draw=black] (296.5,-60) circle (2pt);
\filldraw[fill=white,draw=black] (352,-60) circle (2pt);

\draw[
  decoration={markings, mark=at position 0.5 with {\arrow[scale=1.2]{stealth}}},
  postaction={decorate}
] (352.5,-62.5) -- (390.35,-100);

\draw[dash pattern={on 0.84pt off 2.51pt}]
  (410,-60.5) -- (356,-60.5);

\draw (268.5,-38) -- (268.5,-59) -- (249,-59);

\draw[line width=0.5pt] (340,-59) ++(-3.6:25pt)
  arc[start angle=0, end angle=-36, radius=22pt];

\draw (290.5,-38) node[anchor=north west,inner sep=0.75pt] {$0$};
\draw (345.5,-38) node[anchor=north west,inner sep=0.75pt] {$1$};
\draw (250,-38) node[anchor=north west,inner sep=0.75pt] {$t_{2}$};
\draw (377,-65) node[anchor=north west,inner sep=0.75pt] {$\theta_{0}$};

\draw (396,-105) node[below,inner sep=0.75pt] {$t_1$};
\filldraw[fill=black,draw=black] (390.35,-100) circle (1pt);
\end{tikzpicture}
\quad }}
\right)
\\
&\longrightarrow\quad 
\left(\!\!\!
\vcenter{\hbox{%
\tikzset{every picture/.style={line width=0.75pt}} 

\begin{tikzpicture}[x=0.75pt,y=0.75pt,yscale=0.8,xscale=0.8]
\path[use as bounding box] (245,-30) rectangle (420,-120);

\filldraw[fill=white,draw=black] (296.5,-90) circle (2pt);
\filldraw[fill=white,draw=black] (352,-90) circle (2pt);

\draw[
  decoration={markings, mark=at position 0.6 with {\arrow[scale=1.2]{stealth}}},
  postaction={decorate}
] (353,-88) -- (405.5,-40);

\draw (268.5,-38) -- (268.5,-59) -- (249,-59);

\draw[dash pattern={on 0.84pt off 2.51pt}]
  (410,-90) -- (355,-90);

\draw[line width=0.5pt] (352,-89) ++(-3.6:20pt)
  arc[start angle=0, end angle=42, radius=22pt];

\draw (290.5,-98) node[anchor=north west,inner sep=0.75pt] {$0$};
\draw (346.5,-98) node[anchor=north west,inner sep=0.75pt] {$1$};
\draw (250,-38) node[anchor=north west,inner sep=0.75pt] {$t_{1}$};
\draw (383,-67) node[anchor=north west,inner sep=0.75pt] {$\theta_{-1}$};
\end{tikzpicture}
}}
\right)
\times
\left(\!\!\!
\vcenter{\hbox{%
\tikzset{every picture/.style={line width=0.75pt}} 

\begin{tikzpicture}[x=0.75pt,y=0.75pt,yscale=0.8,xscale=0.8]
\path[use as bounding box] (245,-30) rectangle (400,-120);

\filldraw[fill=white,draw=black] (296.5,-90) circle (2pt);
\filldraw[fill=white,draw=black] (352,-90) circle (2pt);

\draw[dash pattern={on 0.84pt off 2.51pt}]
  (410,-90) -- (355,-90);

\draw[line width=0.5pt] (352,-89) ++(-3.6:20pt)
  arc[start angle=0, end angle=42, radius=22pt];

\draw[
  decoration={markings, mark=at position 0.6 with {\arrow[scale=1.2]{stealth}}},
  postaction={decorate}
] (353,-88) -- (384.5,-59.2);

\draw (268.5,-38) -- (268.5,-59) -- (249,-59);

\draw (290.5,-98) node[anchor=north west,inner sep=0.75pt] {$0$};
\draw (346.5,-98) node[anchor=north west,inner sep=0.75pt] {$1$};
\draw (250,-38) node[anchor=north west,inner sep=0.75pt] {$t_{2}$};
\draw (383,-67) node[anchor=north west,inner sep=0.75pt] {$\theta_{-1}$};

\draw (384.5,-55) node[above,inner sep=0.75pt] {$t_1$};
\filldraw[fill=black,draw=black] (384.5,-59.2) circle (1pt);
\end{tikzpicture}
~ }}
\right)
\end{aligned}
\]
\caption{Deformation of $\Delta_{11}^{\theta_{0}}$ associated with the rotation
	}
\label{fig:deform11}
\end{figure}
By chasing the change of branch of $\Phi(t_1,t_2;x)$, we see that the branch on the resulting deformed domain is the same as the one defined by \eqref{eq:branch3}.
Therefore we have
	\begin{equation}\label{eq:StokesV11}
	V_{11}^{(0)}(x)=V_{11}^{(-1)}(x).
	\end{equation}
\FloatBarrier
Next we consider the deformation of $\Delta_{10}^{\theta_{0}}$. 
In the $t_1$-plane, there is no obstruction to rotating the ray, and we obtain $(t_1:1\to\infty e^{\sqrt{-1}\theta_{-1}})$. 
On the other hand, when we rotate the ray $(t_2:0\to\infty e^{\sqrt{-1}\theta_{0}})$ in the $t_2$-plane, it gets caught on $t_2=1$ (Figure \ref{fig:deform1}).

\begin{figure}[!htbp]
\[
\begin{aligned}
&
\left(\!\!\!
\vcenter{\hbox{%
\tikzset{every picture/.style={line width=0.75pt}} 

\begin{tikzpicture}[x=0.75pt,y=0.75pt,yscale=0.8,xscale=0.8]
\path[use as bounding box] (245,-30) rectangle (420,-120);

\filldraw[fill=white,draw=black] (296.5,-60) circle (2pt);
\filldraw[fill=white,draw=black] (352,-60) circle (2pt);

\draw[
  decoration={markings, mark=at position 0.5 with {\arrow[scale=1.2]{stealth}}},
  postaction={decorate}
] (352.5,-62.5) -- (405.5,-115);

\draw[dash pattern={on 0.84pt off 2.51pt}]
  (410,-60.5) -- (356,-60.5);

\draw (268.5,-38) -- (268.5,-59) -- (249,-59);

\draw[line width=0.5pt] (340,-59) ++(-3.6:25pt)
  arc[start angle=0, end angle=-36, radius=22pt];

\draw (290.5,-38) node[anchor=north west,inner sep=0.75pt] {$0$};
\draw (345.5,-38) node[anchor=north west,inner sep=0.75pt] {$1$};
\draw (250,-38) node[anchor=north west,inner sep=0.75pt] {$t_{1}$};
\draw (377,-65) node[anchor=north west,inner sep=0.75pt] {$\theta_{0}$};
\end{tikzpicture}
}}
\right)
\times
\left(\!\!\!
\vcenter{\hbox{%
\tikzset{every picture/.style={line width=0.75pt}} 

\begin{tikzpicture}[x=0.75pt,y=0.75pt,yscale=0.8,xscale=0.8]
\path[use as bounding box] (245,-30) rectangle (400,-120);

\filldraw[fill=white,draw=black] (296.5,-60) circle (2pt);
\filldraw[fill=white,draw=black] (352,-60) circle (2pt);

\draw[
  decoration={markings, mark=at position 0.5 with {\arrow[scale=1.2]{stealth}}},
  postaction={decorate}
] (297,-62.5) -- (350,-115);

\draw[dash pattern={on 0.84pt off 2.51pt}]
  (325,-60.5) -- (298,-60.5);

\draw (268.5,-38) -- (268.5,-59) -- (249,-59);

\draw[line width=0.5pt] (340-352+296.5,-59) ++(-3.6:25pt)
  arc[start angle=0, end angle=-36, radius=22pt];

\draw (290.5,-38) node[anchor=north west,inner sep=0.75pt] {$0$};
\draw (345.5,-38) node[anchor=north west,inner sep=0.75pt] {$1$};
\draw (250,-38) node[anchor=north west,inner sep=0.75pt] {$t_{2}$};
\draw (377-352+296.5,-65)
  node[anchor=north west,inner sep=0.75pt] {$\theta_{0}$};
\end{tikzpicture}
}}
\right)
\\
&\longrightarrow\quad 
\left(\!\!\!
\vcenter{\hbox{%
\tikzset{every picture/.style={line width=0.75pt}} 

\begin{tikzpicture}[x=0.75pt,y=0.75pt,yscale=0.8,xscale=0.8]
\path[use as bounding box] (245,-30) rectangle (420,-120);

\filldraw[fill=white,draw=black] (296.5,-90) circle (2pt);
\filldraw[fill=white,draw=black] (352,-90) circle (2pt);

\draw[
  decoration={markings, mark=at position 0.6 with {\arrow[scale=1.2]{stealth}}},
  postaction={decorate}
] (353,-88) -- (405.5,-40);

\draw (268.5,-38) -- (268.5,-59) -- (249,-59);

\draw[dash pattern={on 0.84pt off 2.51pt}]
  (410,-90) -- (355,-90);

\draw[line width=0.5pt] (352,-89) ++(-3.6:20pt)
  arc[start angle=0, end angle=42, radius=22pt];

\draw (290.5,-98) node[anchor=north west,inner sep=0.75pt] {$0$};
\draw (346.5,-98) node[anchor=north west,inner sep=0.75pt] {$1$};
\draw (250,-38) node[anchor=north west,inner sep=0.75pt] {$t_{1}$};
\draw (383,-67) node[anchor=north west,inner sep=0.75pt] {$\theta_{-1}$};
\end{tikzpicture}
}}
\right)
\times
\left(\!\!\!
\vcenter{\hbox{%
\tikzset{every picture/.style={line width=0.75pt}} 

\begin{tikzpicture}[x=0.75pt,y=0.75pt,yscale=0.8,xscale=0.8]
\path[use as bounding box] (245,-30) rectangle (400,-120);

\filldraw[fill=white,draw=black] (296.5,-90) circle (2pt);
\filldraw[fill=white,draw=black] (352,-90) circle (2pt);

\draw (298.5,-90) -- (343,-90);

\draw (352,-89) ++(-3.6:7pt)
  arc[start angle=360, end angle=180, radius=7pt];

\draw[
  decoration={markings, mark=at position 0.6 with {\arrow[scale=1.2]{stealth}}},
  postaction={decorate}
] (361,-90) -- (413.5,-38);

\draw[dash pattern={on 0.84pt off 2.51pt}]
  (410,-90) -- (360,-90);

\draw[line width=0.5pt] (355,-89) ++(-3.6:20pt)
  arc[start angle=0, end angle=33, radius=22pt];

\draw (268.5,-38) -- (268.5,-59) -- (249,-59);

\draw (290.5,-70) node[anchor=north west,inner sep=0.75pt] {$0$};
\draw (346,-70) node[anchor=north west,inner sep=0.75pt] {$1$};
\draw (250,-38) node[anchor=north west,inner sep=0.75pt] {$t_{2}$};
\draw (386,-69) node[anchor=north west,inner sep=0.75pt] {$\theta_{-1}$};
\end{tikzpicture}
\quad }}
\right)
\end{aligned}
\]
\caption{Deformation of $\Delta_{10}^{\theta_{0}}$ associated with the rotation
	}
\label{fig:deform1}
\end{figure}
\noindent
We further deform the path in the $t_2$-plane as in Figure \ref{fig:deform2} and divide it into three parts: 
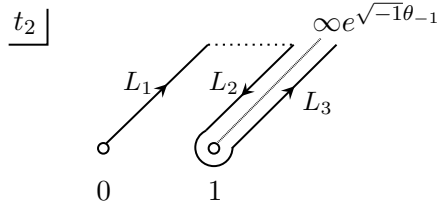
\begin{figure}[!htbp]
\[
\vcenter{\hbox{%
\tikzset{every picture/.style={line width=0.75pt}} 

\begin{tikzpicture}[x=0.75pt,y=0.75pt,yscale=1,xscale=1]

\filldraw[fill=white,draw=black] (296.5,-90) circle (2pt);
\filldraw[fill=white,draw=black] (352,-90) circle (2pt);


\draw (352,-89) ++(-3.6:7pt)
  arc[start angle=360, end angle=115, radius=7pt];

\draw[
  decoration={markings, mark=at position 0.6 with {\arrow[scale=1.2]{stealth}}},
  postaction={decorate}
] (361,-90) -- (413.5,-38);

\draw[
  decoration={markings, mark=at position 0.6 with {\arrow[scale=1.2]{stealth}}},
  postaction={decorate}
] (298,-87.5) -- (349,-38);

\draw[
  decoration={markings, mark=at position 0.6 with {\arrow[scale=1.2]{stealth}}},
  postaction={decorate}
] (392,-38) -- (348,-81);

\draw[dash pattern={on 0.30pt off 0.35pt}]
  (354,-87) -- (405.5,-35);

\draw[dash pattern={on 0.80pt off 2pt}]
  (349,-38) -- (392,-38);

\draw (268.5,-20) -- (268.5,-38) -- (249,-38);

\draw (291.5,-105) node[anchor=north west,inner sep=0.75pt] {$0$};
\draw (347.5,-105) node[anchor=north west,inner sep=0.75pt] {$1$};
\draw (250,-20) node[anchor=north west,inner sep=0.75pt] {$t_{2}$};

\draw (305,-50) node[anchor=north west,inner sep=0.75pt] {\small $L_1$};
\draw (348,-50) node[anchor=north west,inner sep=0.75pt] {\small $L_2$};
\draw (395,-60) node[anchor=north west,inner sep=0.75pt] {\small $L_3$};

\draw (400,-15) node[anchor=north west,inner sep=0.75pt]
  {$\infty e^{\sqrt{-1}\theta_{-1}}$};

\end{tikzpicture}
\quad }}
\]
\caption{Deformation of the path in the $t_2$-plane}
\label{fig:deform2}
\end{figure}
	\begin{equation}
	\begin{aligned}
	L_1: (t_2: 0\to \infty e^{\sqrt{-1}\theta_{-1}}), \quad L_2: (t_2: \infty e^{\sqrt{-1}\theta_{-1}} \to 1), \quad
	L_3: (t_2: 1\to \infty e^{\sqrt{-1}\theta_{-1}}).
	\end{aligned}
	\end{equation}
Here we note that there is a point where $t_1=t_2$ on the product sets $(t_1:1\to\infty e^{\sqrt{-1}\theta_{-1}})\times L_2$ and $(t_1:1\to\infty e^{\sqrt{-1}\theta_{-1}})\times L_3$. 
With this in mind, we further divide $L_2$ and $L_3$ as
	\begin{equation}
	\begin{aligned}
	L_2&=(t_2:\infty e^{\sqrt{-1}\theta_{-1}}\to t_1)+(t_2: t_1\to 1), \\
	L_3&=(t_2:1\to t_1)+(t_2:t_1\to \infty e^{\sqrt{-1}\theta_{-1}}).
	\end{aligned}
	\end{equation}
Consequently, we have 
	\begin{equation}\label{eq:V12345}
	V_{10}^{(0)}(x)=V_{\Delta_{\mathrm{I}}}(x)+V_{\Delta_{\mathrm{\two}}}(x)+V_{\Delta_{\mathrm{\three}}}(x)+V_{\Delta_{\mathrm{\four}}}(x)+V_{\Delta_{\mathrm{V}}}(x),
	\end{equation}
where
	\begin{align}
	\Delta_{\mathrm{I}}&=(t_1:1\to\infty e^{\sqrt{-1}\theta_{-1}})\times L_1, \\
	\Delta_{\mathrm{\two}}&=(t_1:1\to\infty e^{\sqrt{-1}\theta_{-1}})\times (t_2:\infty e^{\sqrt{-1}\theta_{-1}}\to t_1), \\
	\Delta_{\mathrm{\three}}&=(t_1:1\to\infty e^{\sqrt{-1}\theta_{-1}})\times (t_2: t_1\to 1), \\
	\Delta_{\mathrm{\four}}&=(t_1:1\to\infty e^{\sqrt{-1}\theta_{-1}})\times (t_2:1\to t_1), \\
	\Delta_{\mathrm{V}}&=(t_1:1\to\infty e^{\sqrt{-1}\theta_{-1}})\times (t_2:t_1\to \infty e^{\sqrt{-1}\theta_{-1}}). 
	\end{align}
\begin{figure}[!htbp]
\begin{equation}
\begin{aligned}
&\left(\!\!\!
\vcenter{\hbox{%
\tikzset{every picture/.style={line width=0.75pt}} 

\begin{tikzpicture}[x=0.75pt,y=0.75pt,yscale=0.7,xscale=0.7]
\path[use as bounding box] (245,-30) rectangle (420,-120);

\filldraw[fill=white,draw=black] (296.5,-90) circle (2pt);
\filldraw[fill=white,draw=black] (352,-90) circle (2pt);

\draw[
  decoration={markings, mark=at position 0.6 with {\arrow[scale=1.2]{stealth}}},
  postaction={decorate}
] (353,-88) -- (405.5,-40);

\draw (268.5,-38) -- (268.5,-59) -- (249,-59);

\draw (290.5,-98) node[anchor=north west,inner sep=0.75pt] {$0$};
\draw (346.5,-98) node[anchor=north west,inner sep=0.75pt] {$1$};
\draw (248,-38) node[anchor=north west,inner sep=0.75pt] {$t_{1}$};

\end{tikzpicture}
}}
\!\!\right)
\times
\left(\!\!\!
\vcenter{\hbox{%
\tikzset{every picture/.style={line width=0.75pt}} 

\begin{tikzpicture}[x=0.75pt,y=0.75pt,yscale=0.7,xscale=0.7]
\path[use as bounding box] (245,-30) rectangle (420,-120);

\filldraw[fill=white,draw=black] (296.5,-90) circle (2pt);
\filldraw[fill=white,draw=black] (352,-90) circle (2pt);

\draw (352,-89) ++(-3.6:7pt)
  arc[start angle=360, end angle=115, radius=7pt];

\draw[
  decoration={markings, mark=at position 0.6 with {\arrow[scale=1.2]{stealth}}},
  postaction={decorate}
] (361,-90) -- (413.5,-38);

\draw[
  decoration={markings, mark=at position 0.6 with {\arrow[scale=1.2]{stealth}}},
  postaction={decorate}
] (298,-87.5) -- (349,-38);

\draw[
  decoration={markings, mark=at position 0.6 with {\arrow[scale=1.2]{stealth}}},
  postaction={decorate}
] (392,-38) -- (348,-81);

\draw[dash pattern={on 0.30pt off 0.35pt}]
  (354,-87) -- (405.5,-35);

\draw[dash pattern={on 0.80pt off 2pt}]
  (349,-38) -- (392,-38);

\draw (268.5,-38) -- (268.5,-59) -- (249,-59);

\draw (291.5,-105) node[anchor=north west,inner sep=0.75pt] {$0$};
\draw (347.5,-105) node[anchor=north west,inner sep=0.75pt] {$1$};
\draw (248,-38) node[anchor=north west,inner sep=0.75pt] {$t_{2}$};

\end{tikzpicture}
}}
\!\!\right)
 \\
&=\left(\!\!\!
\vcenter{\hbox{%
\tikzset{every picture/.style={line width=0.75pt}} 

\begin{tikzpicture}[x=0.75pt,y=0.75pt,yscale=0.7,xscale=0.7]
\path[use as bounding box] (245,-30) rectangle (420,-120);

\filldraw[fill=white,draw=black] (296.5,-90) circle (2pt);
\filldraw[fill=white,draw=black] (352,-90) circle (2pt);

\draw[
  decoration={markings, mark=at position 0.6 with {\arrow[scale=1.2]{stealth}}},
  postaction={decorate}
] (353,-88) -- (405.5,-40);

\draw (268.5,-38) -- (268.5,-59) -- (249,-59);

\draw (290.5,-98) node[anchor=north west,inner sep=0.75pt] {$0$};
\draw (346.5,-98) node[anchor=north west,inner sep=0.75pt] {$1$};
\draw (248,-38) node[anchor=north west,inner sep=0.75pt] {$t_{1}$};

\end{tikzpicture}
}}
\!\!\right)
\times
\left\{
\left(\!\!\!
\vcenter{\hbox{%
\tikzset{every picture/.style={line width=0.75pt}} 
\begin{tikzpicture}[x=0.75pt,y=0.75pt,yscale=0.7,xscale=0.7]
\path[use as bounding box] (245,-30) rectangle (420,-120);

\filldraw[fill=white,draw=black] (296.5,-90) circle (2pt);
\filldraw[fill=white,draw=black] (352,-90) circle (2pt);

\draw[dash pattern={on 0.80pt off 1pt},opacity=0.5]
  (352,-89) ++(-3.6:7pt)
  arc[start angle=360, end angle=115, radius=7pt];

\draw[dash pattern={on 0.80pt off 1pt},opacity=0.5]
  (361,-90) -- (413.5,-38);

\draw[
  decoration={markings, mark=at position 0.6 with {\arrow[scale=1.2]{stealth}}},
  postaction={decorate}
] (298,-87.5) -- (349,-38);

\draw[dash pattern={on 0.80pt off 1pt},opacity=0.5]
  (392,-38) -- (348,-81);

\draw[dash pattern={on 0.30pt off 0.35pt}]
  (354,-87) -- (405.5,-35);

\draw[dash pattern={on 0.80pt off 1pt},opacity=0.5]
  (349,-38) -- (392,-38);

\draw (268.5,-38) -- (268.5,-59) -- (249,-59);

\draw (291.5,-105) node[anchor=north west,inner sep=0.75pt] {$0$};
\draw (347.5,-105) node[anchor=north west,inner sep=0.75pt] {$1$};
\draw (248,-38) node[anchor=north west,inner sep=0.75pt] {$t_{2}$};

\end{tikzpicture}
}}
\!\!\right)\right. 
+
\left(\!\!\!
\vcenter{\hbox{%
\tikzset{every picture/.style={line width=0.75pt}} 

\begin{tikzpicture}[x=0.75pt,y=0.75pt,yscale=0.7,xscale=0.7]
\path[use as bounding box] (245,-30) rectangle (420,-120);

\filldraw[fill=white,draw=black] (296.5,-90) circle (2pt);
\filldraw[fill=white,draw=black] (352,-90) circle (2pt);

\draw[dash pattern={on 0.80pt off 1pt},opacity=0.5]
  (352,-89) ++(-3.6:7pt)
  arc[start angle=360, end angle=115, radius=7pt];

\draw[dash pattern={on 0.80pt off 1pt},opacity=0.5]
  (361,-90) -- (413.5,-38);

\draw[dash pattern={on 0.80pt off 1pt},opacity=0.5]
  (298,-87.5) -- (349,-38);

\draw[
  decoration={markings, mark=at position 0.6 with {\arrow[scale=1.2]{stealth}}},
  postaction={decorate}
] (392,-38) -- (370,-59.5);

\draw[dash pattern={on 0.80pt off 1pt},opacity=0.5]
  (370,-59.5) -- (348,-81);

\filldraw[fill=black,draw=black] (370,-59.5) circle (1pt);
\draw (367,-48) node[anchor=north east,inner sep=0.75pt] {\small $t_{1}$};

\draw[dash pattern={on 0.30pt off 0.35pt}]
  (354,-87) -- (405.5,-35);

\draw[dash pattern={on 0.80pt off 1pt},opacity=0.5]
  (349,-38) -- (392,-38);

\draw (268.5,-38) -- (268.5,-59) -- (249,-59);

\draw (291.5,-105) node[anchor=north west,inner sep=0.75pt] {$0$};
\draw (347.5,-105) node[anchor=north west,inner sep=0.75pt] {$1$};
\draw (248,-38) node[anchor=north west,inner sep=0.75pt] {$t_{2}$};

\end{tikzpicture}
}}
\!\!\right) \\
&
+\left. \left(\!\!\!
\vcenter{\hbox{%
\tikzset{every picture/.style={line width=0.75pt}} 

\begin{tikzpicture}[x=0.75pt,y=0.75pt,yscale=0.7,xscale=0.7]
\path[use as bounding box] (245,-30) rectangle (420,-120);

\filldraw[fill=white,draw=black] (296.5,-90) circle (2pt);
\filldraw[fill=white,draw=black] (352,-90) circle (2pt);

\draw[dash pattern={on 0.80pt off 1pt},opacity=0.5]
  (352,-89) ++(-3.6:7pt)
  arc[start angle=360, end angle=115, radius=7pt];

\draw[dash pattern={on 0.80pt off 1pt},opacity=0.5]
  (361,-90) -- (413.5,-38);

\draw[dash pattern={on 0.80pt off 1pt},opacity=0.5]
  (298,-87.5) -- (349,-38);

\draw[dash pattern={on 0.80pt off 1pt},opacity=0.5]
  (392,-38) -- (370,-59.5);

\draw[
  decoration={markings, mark=at position 0.6 with {\arrow[scale=1.2]{stealth}}},
  postaction={decorate}
] (370,-59.5) -- (348,-81);

\filldraw[fill=black,draw=black] (370,-59.5) circle (1pt);
\draw (371,-40) node[anchor=north east,inner sep=0.75pt] {\small $t_{1}$};

\draw[dash pattern={on 0.30pt off 0.35pt}]
  (354,-87) -- (405.5,-35);

\draw[dash pattern={on 0.80pt off 1pt},opacity=0.5]
  (349,-38) -- (392,-38);

\draw (268.5,-38) -- (268.5,-59) -- (249,-59);

\draw (291.5,-105) node[anchor=north west,inner sep=0.75pt] {$0$};
\draw (347.5,-105) node[anchor=north west,inner sep=0.75pt] {$1$};
\draw (248,-38) node[anchor=north west,inner sep=0.75pt] {$t_{2}$};

\end{tikzpicture}
}}
\!\!\right)
+
\left(\!\!\!
\vcenter{\hbox{%
\tikzset{every picture/.style={line width=0.75pt}} 

\begin{tikzpicture}[x=0.75pt,y=0.75pt,yscale=0.7,xscale=0.7]
\path[use as bounding box] (245,-30) rectangle (420,-120);

\filldraw[fill=white,draw=black] (296.5,-90) circle (2pt);
\filldraw[fill=white,draw=black] (352,-90) circle (2pt);

\draw[dash pattern={on 0.80pt off 1pt},opacity=0.5]
  (352,-89) ++(-3.6:7pt)
  arc[start angle=360, end angle=115, radius=7pt];

\draw[
  decoration={markings, mark=at position 0.6 with {\arrow[scale=1.2]{stealth}}},
  postaction={decorate}
] (361,-90) -- (387.25,-64);

\draw[dash pattern={on 0.80pt off 1pt},opacity=0.5]
  (387.25,-64) -- (413.5,-38);

\draw[dash pattern={on 0.80pt off 1pt},opacity=0.5]
  (298,-87.5) -- (349,-38);

\draw[dash pattern={on 0.80pt off 1pt},opacity=0.5]
  (392,-38) -- (348,-81);

\filldraw[fill=black,draw=black] (387.25,-64) circle (1pt);
\draw (387.25,-64)
  node[anchor=north west,inner sep=0.75pt] {\small $t_{1}$};

\draw[dash pattern={on 0.30pt off 0.35pt}]
  (354,-87) -- (405.5,-35);

\draw[dash pattern={on 0.80pt off 1pt},opacity=0.5]
  (349,-38) -- (392,-38);

\draw (268.5,-38) -- (268.5,-59) -- (249,-59);

\draw (291.5,-105) node[anchor=north west,inner sep=0.75pt] {$0$};
\draw (347.5,-105) node[anchor=north west,inner sep=0.75pt] {$1$};
\draw (248,-38) node[anchor=north west,inner sep=0.75pt] {$t_{2}$};

\end{tikzpicture}
}}
\!\!\right)
+
\left(\!\!\!
\vcenter{\hbox{%
\tikzset{every picture/.style={line width=0.75pt}} 

\begin{tikzpicture}[x=0.75pt,y=0.75pt,yscale=0.7,xscale=0.7]
\path[use as bounding box] (245,-30) rectangle (420,-120);

\filldraw[fill=white,draw=black] (296.5,-90) circle (2pt);
\filldraw[fill=white,draw=black] (352,-90) circle (2pt);

\draw[dash pattern={on 0.80pt off 1pt},opacity=0.5]
  (352,-89) ++(-3.6:7pt)
  arc[start angle=360, end angle=115, radius=7pt];

\draw[dash pattern={on 0.80pt off 1pt},opacity=0.5]
  (361,-90) -- (387.25,-64);

\draw[
  decoration={markings, mark=at position 0.6 with {\arrow[scale=1.2]{stealth}}},
  postaction={decorate}
] (387.25,-64) -- (413.5,-38);

\draw[dash pattern={on 0.80pt off 1pt},opacity=0.5]
  (298,-87.5) -- (349,-38);

\draw[dash pattern={on 0.80pt off 1pt},opacity=0.5]
  (392,-38) -- (348,-81);

\filldraw[fill=black,draw=black] (387.25,-64) circle (1pt);
\draw (387.25,-64)
  node[anchor=north west,inner sep=0.75pt] {\small $t_{1}$};

\draw[dash pattern={on 0.30pt off 0.35pt}]
  (354,-87) -- (405.5,-35);

\draw[dash pattern={on 0.80pt off 1pt},opacity=0.5]
  (349,-38) -- (392,-38);

\draw (268.5,-38) -- (268.5,-59) -- (249,-59);

\draw (291.5,-105) node[anchor=north west,inner sep=0.75pt] {$0$};
\draw (347.5,-105) node[anchor=north west,inner sep=0.75pt] {$1$};
\draw (248,-38) node[anchor=north west,inner sep=0.75pt] {$t_{2}$};

\end{tikzpicture}
}}
\!\!\right)\right\} 
\\
&=:\Delta_{\mathrm{I}}+\Delta_{\mathrm{\two}}+\Delta_{\mathrm{\three}}+\Delta_{\mathrm{\four}}+\Delta_{\mathrm{V}}
\end{aligned}
\end{equation}
\caption{Partition of the deformed domain}
\label{fig:deform4}
\end{figure}

\noindent
Note that $\Delta_{\mathrm{I}}=\Delta_{10}^{\theta_{-1}}$ and $\Delta_{\mathrm{\three}}=\Delta_{\mathrm{\four}}=\Delta_{11}^{\theta_{-1}}$ as domains.
The branches of the integrand $\Phi(t_1,t_2;x)$ of $V_{\Delta_{\mathrm{I}}}(x),\ldots,V_{\Delta_{\mathrm{V}}}(x)$ on $\Delta_{\mathrm{I}},\ldots,\Delta_{\mathrm{V}}$ are given as follows: 
	\begin{align}
	\label{eq:branchD10}
	&\Delta_{\mathrm{I}}:
		\begin{cases}
		\arg t_1=0 \quad \text{at $t_1=1$}, \\
		\arg(1-t_1)=\theta_{-1}+\pi, \\
		\arg t_2=\theta_{-1}, \\
		\arg(1-t_2)=0 \quad \text{at $t_2=0$}, \\
		\arg(t_1-t_2)=0 \quad \text{at $(t_1,t_2)=(1,0)$},
		\end{cases} 
	\quad
	&\Delta_{\mathrm{\two}}:
		\begin{cases}
		\arg t_1=0 \quad \text{at $t_1=1$}, \\
		\arg(1-t_1)=\theta_{-1}+\pi, \\
		\arg t_2=0\quad \text{at $t_2=1$}, \\
		\arg(1-t_2)=\theta_{-1}-\pi, \\
		\arg(t_1-t_2)=\theta_{-1}-\pi,
		\end{cases} 
	\\
	&\Delta_{\mathrm{\three}}:
		\begin{cases}
		\arg t_1=0 \quad \text{at $t_1=1$}, \\
		\arg(1-t_1)=\theta_{-1}+\pi, \\
		\arg t_2=0 \quad \text{at $t_2=1$}, \\
		\arg(1-t_2)=\theta_{-1}-\pi , \\
		\arg(t_1-t_2)=\theta_{-1}\phantom{ \text{at $(t_1,t_2)=(1,0)$}}
		\end{cases} 
	\quad
	&\Delta_{\mathrm{\four}}:
		\begin{cases}
		\arg t_1=0\quad \text{at $t_1=1$}, \\
		\arg(1-t_1)=\theta_{-1}+\pi, \\
		\arg t_2=0 \quad \text{at $t_2=1$}, \\
		\arg(1-t_2)=\theta_{-1}+\pi, \\
		\arg(t_1-t_2)=\theta_{-1},
		\end{cases} 
	\\
	&\Delta_{\mathrm{V}}:
		\begin{cases}
		\arg t_1=0\quad \text{at $t_1=1$}, \\
		\arg(1-t_1)=\theta_{-1}+\pi, \\
		\arg t_2=0 \quad \text{at $t_2=1$}, \\
		\arg(1-t_2)=\theta_{-1}+\pi, \\
		\arg(t_1-t_2)=\theta_{-1}+\pi.
		\end{cases} 
	\end{align}
Comparing these branches with \eqref{eq:branch2} and \eqref{eq:branch3}, we have
	\begin{equation}\label{eq:V134}
	\begin{aligned}
	&V_{\Delta_{\mathrm{I}}}(x)=V_{10}^{(-1)}(x), 
	\quad 
	V_{\Delta_{\mathrm{\three}}}(x)=-e(-2b)V_{11}^{(-1)}(x), \\
	&V_{\Delta_{\mathrm{\four}}}(x)=e(-2b)e(2b)V_{11}^{(-1)}(x)=V_{11}^{(-1)}(x). 
	\end{aligned}
	\end{equation}
Let us examine $V_{\Delta_{\mathrm{\two}}}(x)$ and $V_{\Delta_{\mathrm{V}}}(x)$. 
Comparing the branches, we have
	\begin{equation}\label{eq:V5}
		V_{\Delta_{\mathrm{\two}}}=-e(-2b)e(-2g)V_{\Delta_{\mathrm{V}}}=-e(-2b-2g)V_{\Delta_{\mathrm{V}}}.
	\end{equation}
For $V_{\Delta_{\mathrm{V}}}$, we have the following. 
	\begin{lemma}\label{lem:DeltaV}
	$
	V_{\Delta_{\mathrm{V}}}(x)=e(g)V_{11}^{(-1)}(x).
	$
	\end{lemma}
\begin{proof}
We write the integrand $\Phi$ of $V_{\Delta_{\mathrm{V}}}(x)$ as $\Phi_{\Delta_{\mathrm{V}}}$
and consider the scalar integral
	\begin{equation}
	\begin{aligned}
	\int_{\Delta_{\mathrm{V}}}\Phi_{\Delta_{\mathrm{V}}}(t_1,t_2;x)\, dt_1\wedge dt_2 
	&=\int_{1}^{\infty e^{\sqrt{-1}\theta_{-1}}}dt_1\int_{t_1}^{\infty e^{\sqrt{-1}\theta_{-1}}}dt_2 \,\Phi_{\Delta_{\mathrm{V}}}(t_1,t_2;x) \\
	&=\int_{1}^{\infty e^{\sqrt{-1}\theta_{-1}}}dt_2\int_{1}^{t_2}dt_1\,\Phi_{\Delta_{\mathrm{V}}}(t_1,t_2;x).
	\end{aligned}
	\end{equation}
To obtain the last expression, we interchanged the order of integration. 
Since the variables in the last iterated integral are dummy variables, we may interchange their names and obtain
	\begin{equation}
	\begin{aligned}
	\int_{1}^{\infty e^{\sqrt{-1}\theta_{-1}}}dt_2\int_{1}^{t_2}dt_1\,\Phi_{\Delta_{\mathrm{V}}}(t_1,t_2;x) &= \int_{1}^{\infty e^{\sqrt{-1}\theta_{-1}}}dt_1\int_{1}^{t_1}dt_2\,\Phi_{\Delta_{\mathrm{V}}}(t_2,t_1;x) 
	\\
	&=\int_{\Delta_{11}^{\theta_{-1}}}\Phi_{\Delta_{\mathrm{V}}}(t_2,t_1;x)\,dt_1\wedge dt_2.
	\end{aligned}
	\end{equation}
The branch of the integrand $\Phi_{\Delta_{\mathrm{V}}}(t_2,t_1;x)$ of the last integral on $\Delta_{11}^{\theta_{-1}}$ is inherited from that of $\Phi_{\Delta_{\mathrm{V}}}(t_1,t_2;x)$ of the first integral $V_{\Delta_{\mathrm{V}}}$ on $\Delta_{\mathrm{V}}$:
\begin{equation}
		\begin{cases}
		\arg t_2=0 \quad \text{at $t_2=1$}, \\
		\arg(1-t_2)=\theta_{-1}+\pi, \\
		\arg t_1=0\quad \text{at $t_1=1$}, \\
		\arg(1-t_1)=\theta_{-1}+\pi, \\
		\arg(t_2-t_1)=\theta_{-1}+\pi.
		\end{cases} 
\end{equation} 
Comparing these branches with \eqref{eq:branch3}, we have
	\begin{equation}
	\begin{aligned}
	\Phi_{\Delta_{\mathrm{V}}}(t_2,t_1;x)&=e(g)\Phi_{\Delta_{11}^{\theta_{-1}}}(t_1,t_2;x),
	\end{aligned}
	\end{equation}
where $\Phi_{\Delta_{11}^{\theta_{-1}}}(t_1,t_2;x)$ denotes the function $\Phi(t_1,t_2;x)$ whose branch on $\Delta_{11}^{\theta_{-1}}$ is given by \eqref{eq:branch3}. 
Thus, we have
	\begin{equation}\label{eq:symscalar}
		\int_{\Delta_{\mathrm{V}}}\Phi_{\Delta_{\mathrm{V}}}(t_1,t_2;x)\, dt_1\wedge dt_2
		=e(g)\int_{\Delta_{11}^{\theta_{-1}}}\Phi_{\Delta_{11}^{\theta_{-1}}}(t_1,t_2;x)\,dt_1\wedge dt_2.
	\end{equation}
Since the vector $\vec{\varphi}$ consists of rational $2$-forms whose coefficients are symmetric with respect to $t_1$ and $t_2$, the same argument applies componentwise to the integral with $\vec{\varphi}$.
This completes the proof. 
\end{proof}

Summarizing the above arguments, we obtain
	\begin{equation}\label{eq:StokesV10}
	\begin{aligned}
	V_{10}^{(0)}(x)
	=V_{10}^{(-1)}(x)+\left(1-e(-2b-g)\right)\left(1+e(g)\right)V_{11}^{(-1)}(x).
	\end{aligned}
	\end{equation}
\FloatBarrier
Finally, we consider the deformation of $\Delta_{00}^{\theta_{0}}$.
In this case, the points $1$ in the $t_1$- and $t_2$-planes obstruct the rotation (Figure \ref{fig:deform00-1}). 

\begin{figure}[!htbp]
\[
\begin{aligned}
&
\left(\!\!\!
\vcenter{\hbox{%
\tikzset{every picture/.style={line width=0.75pt}} 

\begin{tikzpicture}[x=0.75pt,y=0.75pt,yscale=0.8,xscale=0.8]
\path[use as bounding box] (245,-30) rectangle (420,-120);

\filldraw[fill=white,draw=black] (296.5,-60) circle (2pt);
\filldraw[fill=white,draw=black] (352,-60) circle (2pt);

\draw[
  decoration={markings, mark=at position 0.5 with {\arrow[scale=1.2]{stealth}}},
  postaction={decorate}
] (297,-62.5) -- (350,-115);

\draw[dash pattern={on 0.84pt off 2.51pt}]
  (325,-60.5) -- (298,-60.5);

\draw (268.5,-38) -- (268.5,-59) -- (249,-59);

\draw[line width=0.5pt] (340-352+296.5,-59) ++(-3.6:25pt)
  arc[start angle=0, end angle=-36, radius=22pt];

\draw (290.5,-38) node[anchor=north west,inner sep=0.75pt] {$0$};
\draw (345.5,-38) node[anchor=north west,inner sep=0.75pt] {$1$};
\draw (250,-38) node[anchor=north west,inner sep=0.75pt] {$t_{1}$};
\draw (321.5,-65) node[anchor=north west,inner sep=0.75pt] {$\theta_{0}$};

\end{tikzpicture}
}}
\right)
\times
\left(\!\!\!
\vcenter{\hbox{%
\tikzset{every picture/.style={line width=0.75pt}} 

\begin{tikzpicture}[x=0.75pt,y=0.75pt,yscale=0.8,xscale=0.8]
\path[use as bounding box] (245,-30) rectangle (400,-120);

\filldraw[fill=white,draw=black] (296.5,-60) circle (2pt);
\filldraw[fill=white,draw=black] (352,-60) circle (2pt);

\draw[
  decoration={markings, mark=at position 0.5 with {\arrow[scale=1.2]{stealth}}},
  postaction={decorate}
] (297,-62.5) -- (328.8,-94);

\draw[dash pattern={on 0.84pt off 2.51pt}]
  (325,-60.5) -- (298,-60.5);

\draw (268.5,-38) -- (268.5,-59) -- (249,-59);

\draw[line width=0.5pt] (284.5,-59) ++(-3.6:25pt)
  arc[start angle=0, end angle=-36, radius=22pt];

\draw (290.5,-38) node[anchor=north west,inner sep=0.75pt] {$0$};
\draw (345.5,-38) node[anchor=north west,inner sep=0.75pt] {$1$};
\draw (250,-38) node[anchor=north west,inner sep=0.75pt] {$t_{2}$};
\draw (377-352+296.5,-65)
  node[anchor=north west,inner sep=0.75pt] {$\theta_{0}$};

\draw (336,-95) node[below,inner sep=0.75pt] {$t_1$};
\filldraw[fill=black,draw=black] (328.8,-94) circle (1pt);

\end{tikzpicture}
\quad }}
\right)
\\
&\longrightarrow\quad 
\left(\!\!\!
\vcenter{\hbox{%
\tikzset{every picture/.style={line width=0.75pt}} 

\begin{tikzpicture}[x=0.75pt,y=0.75pt,yscale=0.8,xscale=0.8]
\path[use as bounding box] (245,-30) rectangle (420,-120);

\filldraw[fill=white,draw=black] (296.5,-90) circle (2pt);
\filldraw[fill=white,draw=black] (352,-90) circle (2pt);

\draw (298.5,-90) -- (343,-90);

\draw (352,-89) ++(-3.6:7pt)
  arc[start angle=360, end angle=180, radius=7pt];

\draw[
  decoration={markings, mark=at position 0.6 with {\arrow[scale=1.2]{stealth}}},
  postaction={decorate}
] (361,-90) -- (413.5,-38);

\draw[dash pattern={on 0.84pt off 2.51pt}]
  (410,-90) -- (360,-90);

\draw[line width=0.5pt] (355,-89) ++(-3.6:20pt)
  arc[start angle=0, end angle=33, radius=22pt];

\draw (268.5,-38) -- (268.5,-59) -- (249,-59);

\draw (290.5,-70) node[anchor=north west,inner sep=0.75pt] {$0$};
\draw (346,-70) node[anchor=north west,inner sep=0.75pt] {$1$};
\draw (250,-38) node[anchor=north west,inner sep=0.75pt] {$t_{1}$};
\draw (386,-69) node[anchor=north west,inner sep=0.75pt] {$\theta_{-1}$};

\end{tikzpicture}
}}
\right)
\times
\left(\!\!\!
\vcenter{\hbox{%
\tikzset{every picture/.style={line width=0.75pt}} 

\begin{tikzpicture}[x=0.75pt,y=0.75pt,yscale=0.8,xscale=0.8]
\path[use as bounding box] (245,-30) rectangle (400,-120);

\filldraw[fill=white,draw=black] (296.5,-90) circle (2pt);
\filldraw[fill=white,draw=black] (352,-90) circle (2pt);

\draw (298.5,-90) -- (343,-90);

\draw (352,-89) ++(-3.6:7pt)
  arc[start angle=360, end angle=180, radius=7pt];

\draw[
  decoration={markings, mark=at position 0.6 with {\arrow[scale=1.2]{stealth}}},
  postaction={decorate}
] (361,-90) -- (391,-60.2857);

\draw[dash pattern={on 0.84pt off 2.51pt}]
  (410,-90) -- (360,-90);

\draw[line width=0.5pt] (355,-89) ++(-3.6:20pt)
  arc[start angle=0, end angle=33, radius=22pt];

\draw (268.5,-38) -- (268.5,-59) -- (249,-59);

\draw (290.5,-70) node[anchor=north west,inner sep=0.75pt] {$0$};
\draw (346,-70) node[anchor=north west,inner sep=0.75pt] {$1$};
\draw (250,-38) node[anchor=north west,inner sep=0.75pt] {$t_{2}$};
\draw (386,-69) node[anchor=north west,inner sep=0.75pt] {$\theta_{-1}$};

\draw (383,-57) node[above,inner sep=0.75pt] {$t_1$};
\filldraw[fill=black,draw=black] (391,-60.2857) circle (1pt);

\end{tikzpicture}
\quad }}
\right)
\end{aligned}
\]
\caption{Deformation of $\Delta_{00}^{\theta_{0}}$ associated with the rotation
	}
\label{fig:deform00-1}
\end{figure}
\noindent
As in the case of $\Delta_{10}^{\theta_{0}}$, we further deform the domain and decompose it into seven domains as in Figure \ref{fig:deform00-2}. 
\begin{figure}[!htbp]
\begin{equation}
\begin{aligned}
&\left(\!\!\!
\vcenter{\hbox{%
\tikzset{every picture/.style={line width=0.75pt}} 

\begin{tikzpicture}[x=0.75pt,y=0.75pt,yscale=0.7,xscale=0.7]
\path[use as bounding box] (245,-30) rectangle (420,-120);

\filldraw[fill=white,draw=black] (296.5,-90) circle (2pt);
\filldraw[fill=white,draw=black] (352,-90) circle (2pt);

\draw (352,-89) ++(-3.6:7pt)
  arc[start angle=360, end angle=115, radius=7pt];

\draw[
  decoration={markings, mark=at position 0.6 with {\arrow[scale=1.2]{stealth}}},
  postaction={decorate}
] (361,-90) -- (413.5,-38);

\draw[
  decoration={markings, mark=at position 0.6 with {\arrow[scale=1.2]{stealth}}},
  postaction={decorate}
] (298,-87.5) -- (349,-38);

\draw[
  decoration={markings, mark=at position 0.6 with {\arrow[scale=1.2]{stealth}}},
  postaction={decorate}
] (392,-38) -- (348,-81);

\draw[dash pattern={on 0.30pt off 0.35pt}]
  (354,-87) -- (405.5,-35);

\draw[dash pattern={on 0.80pt off 2pt}]
  (349,-38) -- (392,-38);

\draw (268.5,-38) -- (268.5,-59) -- (249,-59);

\draw (291.5,-105) node[anchor=north west,inner sep=0.75pt] {$0$};
\draw (347.5,-105) node[anchor=north west,inner sep=0.75pt] {$1$};
\draw (248,-38) node[anchor=north west,inner sep=0.75pt] {$t_{1}$};

\end{tikzpicture}
}}
\!\!\right)
\times
\left(\!\!\!
\vcenter{\hbox{%
\tikzset{every picture/.style={line width=0.75pt}} 

\begin{tikzpicture}[x=0.75pt,y=0.75pt,yscale=0.7,xscale=0.7]
\path[use as bounding box] (245,-30) rectangle (420,-120);

\filldraw[fill=white,draw=black] (296.5,-90) circle (2pt);
\filldraw[fill=white,draw=black] (352,-90) circle (2pt);

\draw (352,-89) ++(-3.6:7pt)
  arc[start angle=360, end angle=115, radius=7pt];

\draw[
  decoration={markings, mark=at position 0.6 with {\arrow[scale=1.2]{stealth}}},
  postaction={decorate}
] (361,-90) -- (391,-60.2857);

\draw[
  decoration={markings, mark=at position 0.6 with {\arrow[scale=1.2]{stealth}}},
  postaction={decorate}
] (298,-87.5) -- (349,-38);

\draw[
  decoration={markings, mark=at position 0.6 with {\arrow[scale=1.2]{stealth}}},
  postaction={decorate}
] (392,-38) -- (348,-81);

\draw[dash pattern={on 0.30pt off 0.35pt}]
  (354,-87) -- (405.5,-35);

\draw[dash pattern={on 0.80pt off 2pt}]
  (349,-38) -- (392,-38);

\draw (268.5,-38) -- (268.5,-59) -- (249,-59);

\draw (291.5,-105) node[anchor=north west,inner sep=0.75pt] {$0$};
\draw (347.5,-105) node[anchor=north west,inner sep=0.75pt] {$1$};
\draw (248,-38) node[anchor=north west,inner sep=0.75pt] {$t_{2}$};

\filldraw[fill=black,draw=black] (391,-60.2857) circle (1pt);
\draw (397,-60.2857) node[right,inner sep=0.75pt] {$t_{1}$};

\end{tikzpicture}
}}
\!\!\right)
 \\
&=
\left(\!\!\!
\vcenter{\hbox{%
\tikzset{every picture/.style={line width=0.75pt}} 

\begin{tikzpicture}[x=0.75pt,y=0.75pt,yscale=0.7,xscale=0.7]
\path[use as bounding box] (245,-30) rectangle (420,-120);

\filldraw[fill=white,draw=black] (296.5,-90) circle (2pt);
\filldraw[fill=white,draw=black] (352,-90) circle (2pt);

\draw[dash pattern={on 0.80pt off 1pt},opacity=0.5]
  (352,-89) ++(-3.6:7pt)
  arc[start angle=360, end angle=115, radius=7pt];

\draw[dash pattern={on 0.80pt off 1pt},opacity=0.5]
  (361,-90) -- (413.5,-38);

\draw[
  decoration={markings, mark=at position 0.6 with {\arrow[scale=1.2]{stealth}}},
  postaction={decorate}
] (298,-87.5) -- (349,-38);

\draw[dash pattern={on 0.80pt off 1pt},opacity=0.5]
  (392,-38) -- (348,-81);

\draw[dash pattern={on 0.30pt off 0.35pt}]
  (354,-87) -- (405.5,-35);

\draw[dash pattern={on 0.80pt off 1pt},opacity=0.5]
  (349,-38) -- (392,-38);

\draw (268.5,-38) -- (268.5,-59) -- (249,-59);

\draw (291.5,-105) node[anchor=north west,inner sep=0.75pt] {$0$};
\draw (347.5,-105) node[anchor=north west,inner sep=0.75pt] {$1$};
\draw (248,-38) node[anchor=north west,inner sep=0.75pt] {$t_{1}$};

\end{tikzpicture}
}}
\!\!\right)
\times
\left(\!\!\!
\vcenter{\hbox{%
\tikzset{every picture/.style={line width=0.75pt}} 

\begin{tikzpicture}[x=0.75pt,y=0.75pt,yscale=0.7,xscale=0.7]
\path[use as bounding box] (245,-30) rectangle (420,-120);

\filldraw[fill=white,draw=black] (296.5,-90) circle (2pt);
\filldraw[fill=white,draw=black] (352,-90) circle (2pt);

\draw[dash pattern={on 0.80pt off 1pt},opacity=0.5]
  (352,-89) ++(-3.6:7pt)
  arc[start angle=360, end angle=115, radius=7pt];

\draw[dash pattern={on 0.80pt off 1pt},opacity=0.5]
  (361,-90) -- (413.5,-38);

\draw[
  decoration={markings, mark=at position 0.6 with {\arrow[scale=1.2]{stealth}}},
  postaction={decorate}
] (298,-87.5) -- (323.5,-62.75);

\draw[dash pattern={on 0.80pt off 1pt},opacity=0.5]
  (323.5,-62.75) -- (349,-38);

\draw[dash pattern={on 0.80pt off 1pt},opacity=0.5]
  (392,-38) -- (348,-81);

\draw[dash pattern={on 0.30pt off 0.35pt}]
  (354,-87) -- (405.5,-35);

\draw[dash pattern={on 0.80pt off 1pt},opacity=0.5]
  (349,-38) -- (392,-38);

\draw (268.5,-38) -- (268.5,-59) -- (249,-59);

\draw (291.5,-105) node[anchor=north west,inner sep=0.75pt] {$0$};
\draw (347.5,-105) node[anchor=north west,inner sep=0.75pt] {$1$};
\draw (248,-38) node[anchor=north west,inner sep=0.75pt] {$t_{2}$};

\filldraw[fill=black,draw=black] (323.5,-62.75) circle (1pt);
\draw (323.5,-53) node[left,inner sep=0.75pt] {\small $t_{1}$};

\end{tikzpicture}
}}
\!\!\right) 
 \\
&+\left(\!\!\!
\vcenter{\hbox{%
\tikzset{every picture/.style={line width=0.75pt}} 

\begin{tikzpicture}[x=0.75pt,y=0.75pt,yscale=0.7,xscale=0.7]
\path[use as bounding box] (245,-30) rectangle (420,-120);

\filldraw[fill=white,draw=black] (296.5,-90) circle (2pt);
\filldraw[fill=white,draw=black] (352,-90) circle (2pt);

\draw[dash pattern={on 0.80pt off 1pt},opacity=0.5]
  (352,-89) ++(-3.6:7pt)
  arc[start angle=360, end angle=115, radius=7pt];

\draw[dash pattern={on 0.80pt off 1pt},opacity=0.5]
  (361,-90) -- (413.5,-38);

\draw[dash pattern={on 0.80pt off 1pt},opacity=0.5]
  (298,-87.5) -- (349,-38);

\draw[
  decoration={markings, mark=at position 0.6 with {\arrow[scale=1.2]{stealth}}},
  postaction={decorate}
] (392,-38) -- (348,-81);

\draw[dash pattern={on 0.30pt off 0.35pt}]
  (354,-87) -- (405.5,-35);

\draw[dash pattern={on 0.80pt off 1pt},opacity=0.5]
  (349,-38) -- (392,-38);

\draw (268.5,-38) -- (268.5,-59) -- (249,-59);

\draw (291.5,-105) node[anchor=north west,inner sep=0.75pt] {$0$};
\draw (347.5,-105) node[anchor=north west,inner sep=0.75pt] {$1$};
\draw (248,-38) node[anchor=north west,inner sep=0.75pt] {$t_{1}$};

\end{tikzpicture}
}}
\!\!\right)
\times \left\{ 
\left(\!\!\!
\vcenter{\hbox{%
\tikzset{every picture/.style={line width=0.75pt}} 

\begin{tikzpicture}[x=0.75pt,y=0.75pt,yscale=0.7,xscale=0.7]
\path[use as bounding box] (245,-30) rectangle (420,-120);

\filldraw[fill=white,draw=black] (296.5,-90) circle (2pt);
\filldraw[fill=white,draw=black] (352,-90) circle (2pt);

\draw[dash pattern={on 0.80pt off 1pt},opacity=0.5]
  (352,-89) ++(-3.6:7pt)
  arc[start angle=360, end angle=115, radius=7pt];

\draw[dash pattern={on 0.80pt off 1pt},opacity=0.5]
  (361,-90) -- (413.5,-38);

\draw[
  decoration={markings, mark=at position 0.6 with {\arrow[scale=1.2]{stealth}}},
  postaction={decorate}
] (298,-87.5) -- (349,-38);

\draw[dash pattern={on 0.80pt off 1pt},opacity=0.5]
  (392,-38) -- (348,-81);

\draw[dash pattern={on 0.30pt off 0.35pt}]
  (354,-87) -- (405.5,-35);

\draw[dash pattern={on 0.80pt off 1pt},opacity=0.5]
  (349,-38) -- (392,-38);

\draw (268.5,-38) -- (268.5,-59) -- (249,-59);

\draw (291.5,-105) node[anchor=north west,inner sep=0.75pt] {$0$};
\draw (347.5,-105) node[anchor=north west,inner sep=0.75pt] {$1$};
\draw (248,-38) node[anchor=north west,inner sep=0.75pt] {$t_{2}$};

\end{tikzpicture}
}}
\!\!\right)
+
\left(\!\!\!
\vcenter{\hbox{%
\tikzset{every picture/.style={line width=0.75pt}} 

\begin{tikzpicture}[x=0.75pt,y=0.75pt,yscale=0.7,xscale=0.7]
\path[use as bounding box] (245,-30) rectangle (420,-120);

\filldraw[fill=white,draw=black] (296.5,-90) circle (2pt);
\filldraw[fill=white,draw=black] (352,-90) circle (2pt);

\draw[dash pattern={on 0.80pt off 1pt},opacity=0.5]
  (352,-89) ++(-3.6:7pt)
  arc[start angle=360, end angle=115, radius=7pt];

\draw[dash pattern={on 0.80pt off 1pt},opacity=0.5]
  (361,-90) -- (413.5,-38);

\draw[dash pattern={on 0.80pt off 1pt},opacity=0.5]
  (298,-87.5) -- (349,-38);

\draw[
  decoration={markings, mark=at position 0.6 with {\arrow[scale=1.2]{stealth}}},
  postaction={decorate}
] (392,-38) -- (370,-59.5);

\draw[dash pattern={on 0.80pt off 1pt},opacity=0.5]
  (370,-59.5) -- (348,-81);

\filldraw[fill=black,draw=black] (370,-59.5) circle (1pt);
\draw (367,-48) node[anchor=north east,inner sep=0.75pt] {\small $t_{1}$};

\draw[dash pattern={on 0.30pt off 0.35pt}]
  (354,-87) -- (405.5,-35);

\draw[dash pattern={on 0.80pt off 1pt},opacity=0.5]
  (349,-38) -- (392,-38);

\draw (268.5,-38) -- (268.5,-59) -- (249,-59);

\draw (291.5,-105) node[anchor=north west,inner sep=0.75pt] {$0$};
\draw (347.5,-105) node[anchor=north west,inner sep=0.75pt] {$1$};
\draw (248,-38) node[anchor=north west,inner sep=0.75pt] {$t_{2}$};

\end{tikzpicture}
}}
\!\!\right)\right\} 
\\
&+\left(\!\!\!
\vcenter{\hbox{%
\tikzset{every picture/.style={line width=0.75pt}} 

\begin{tikzpicture}[x=0.75pt,y=0.75pt,yscale=0.7,xscale=0.7]
\path[use as bounding box] (245,-30) rectangle (420,-120);

\filldraw[fill=white,draw=black] (296.5,-90) circle (2pt);
\filldraw[fill=white,draw=black] (352,-90) circle (2pt);

\draw[dash pattern={on 0.80pt off 1pt},opacity=0.5]
  (352,-89) ++(-3.6:7pt)
  arc[start angle=360, end angle=115, radius=7pt];

\draw[
  decoration={markings, mark=at position 0.6 with {\arrow[scale=1.2]{stealth}}},
  postaction={decorate}
] (361,-90) -- (413.5,-38);

\draw[dash pattern={on 0.80pt off 1pt},opacity=0.5]
  (298,-87.5) -- (349,-38);

\draw[dash pattern={on 0.80pt off 1pt},opacity=0.5]
  (392,-38) -- (348,-81);

\draw[dash pattern={on 0.30pt off 0.35pt}]
  (354,-87) -- (405.5,-35);

\draw[dash pattern={on 0.80pt off 1pt},opacity=0.5]
  (349,-38) -- (392,-38);

\draw (268.5,-38) -- (268.5,-59) -- (249,-59);

\draw (291.5,-105) node[anchor=north west,inner sep=0.75pt] {$0$};
\draw (347.5,-105) node[anchor=north west,inner sep=0.75pt] {$1$};
\draw (248,-38) node[anchor=north west,inner sep=0.75pt] {$t_{1}$};

\end{tikzpicture}
}}
\!\!\right)
\times \left\{ 
\left(\!\!\!
\vcenter{\hbox{%
\tikzset{every picture/.style={line width=0.75pt}} 

\begin{tikzpicture}[x=0.75pt,y=0.75pt,yscale=0.7,xscale=0.7]
\path[use as bounding box] (245,-30) rectangle (420,-120);

\filldraw[fill=white,draw=black] (296.5,-90) circle (2pt);
\filldraw[fill=white,draw=black] (352,-90) circle (2pt);

\draw[dash pattern={on 0.80pt off 1pt},opacity=0.5]
  (352,-89) ++(-3.6:7pt)
  arc[start angle=360, end angle=115, radius=7pt];

\draw[dash pattern={on 0.80pt off 1pt},opacity=0.5]
  (361,-90) -- (413.5,-38);

\draw[
  decoration={markings, mark=at position 0.6 with {\arrow[scale=1.2]{stealth}}},
  postaction={decorate}
] (298,-87.5) -- (349,-38);

\draw[dash pattern={on 0.80pt off 1pt},opacity=0.5]
  (392,-38) -- (348,-81);

\draw[dash pattern={on 0.30pt off 0.35pt}]
  (354,-87) -- (405.5,-35);

\draw[dash pattern={on 0.80pt off 1pt},opacity=0.5]
  (349,-38) -- (392,-38);

\draw (268.5,-38) -- (268.5,-59) -- (249,-59);

\draw (291.5,-105) node[anchor=north west,inner sep=0.75pt] {$0$};
\draw (347.5,-105) node[anchor=north west,inner sep=0.75pt] {$1$};
\draw (248,-38) node[anchor=north west,inner sep=0.75pt] {$t_{2}$};

\end{tikzpicture}
}}
\!\!\right)
+
\left(\!\!\!
\vcenter{\hbox{%
\tikzset{every picture/.style={line width=0.75pt}} 

\begin{tikzpicture}[x=0.75pt,y=0.75pt,yscale=0.7,xscale=0.7]
\path[use as bounding box] (245,-30) rectangle (420,-120);

\filldraw[fill=white,draw=black] (296.5,-90) circle (2pt);
\filldraw[fill=white,draw=black] (352,-90) circle (2pt);

\draw[dash pattern={on 0.80pt off 1pt},opacity=0.5]
  (352,-89) ++(-3.6:7pt)
  arc[start angle=360, end angle=115, radius=7pt];

\draw[dash pattern={on 0.80pt off 1pt},opacity=0.5]
  (361,-90) -- (413.5,-38);

\draw[dash pattern={on 0.80pt off 1pt},opacity=0.5]
  (298,-87.5) -- (349,-38);

\draw[
  decoration={markings, mark=at position 0.6 with {\arrow[scale=1.2]{stealth}}},
  postaction={decorate}
] (392,-38) -- (370,-59.5);

\draw[dash pattern={on 0.80pt off 1pt},opacity=0.5]
  (370,-59.5) -- (348,-81);

\filldraw[fill=black,draw=black] (370,-59.5) circle (1pt);
\draw (367,-48) node[anchor=north east,inner sep=0.75pt] {\small $t_{1}$};

\draw[dash pattern={on 0.30pt off 0.35pt}]
  (354,-87) -- (405.5,-35);

\draw[dash pattern={on 0.80pt off 1pt},opacity=0.5]
  (349,-38) -- (392,-38);

\draw (268.5,-38) -- (268.5,-59) -- (249,-59);

\draw (291.5,-105) node[anchor=north west,inner sep=0.75pt] {$0$};
\draw (347.5,-105) node[anchor=north west,inner sep=0.75pt] {$1$};
\draw (248,-38) node[anchor=north west,inner sep=0.75pt] {$t_{2}$};

\end{tikzpicture}
}}
\!\!\right)\right.
\\
&\hspace{9\baselineskip} +\left. 
\left(\!\!\!
\vcenter{\hbox{%
\tikzset{every picture/.style={line width=0.75pt}} 

\begin{tikzpicture}[x=0.75pt,y=0.75pt,yscale=0.7,xscale=0.7]
\path[use as bounding box] (245,-30) rectangle (420,-120);

\filldraw[fill=white,draw=black] (296.5,-90) circle (2pt);
\filldraw[fill=white,draw=black] (352,-90) circle (2pt);

\draw[dash pattern={on 0.80pt off 1pt},opacity=0.5]
  (352,-89) ++(-3.6:7pt)
  arc[start angle=360, end angle=115, radius=7pt];

\draw[dash pattern={on 0.80pt off 1pt},opacity=0.5]
  (361,-90) -- (413.5,-38);

\draw[dash pattern={on 0.80pt off 1pt},opacity=0.5]
  (298,-87.5) -- (349,-38);

\draw[dash pattern={on 0.80pt off 1pt},opacity=0.5]
  (392,-38) -- (370,-59.5);

\draw[
  decoration={markings, mark=at position 0.6 with {\arrow[scale=1.2]{stealth}}},
  postaction={decorate}
] (370,-59.5) -- (348,-81);

\filldraw[fill=black,draw=black] (370,-59.5) circle (1pt);
\draw (367,-44) node[anchor=north east,inner sep=0.75pt] {\small $t_{1}$};

\draw[dash pattern={on 0.30pt off 0.35pt}]
  (354,-87) -- (405.5,-35);

\draw[dash pattern={on 0.80pt off 1pt},opacity=0.5]
  (349,-38) -- (392,-38);

\draw (268.5,-38) -- (268.5,-59) -- (249,-59);

\draw (291.5,-105) node[anchor=north west,inner sep=0.75pt] {$0$};
\draw (347.5,-105) node[anchor=north west,inner sep=0.75pt] {$1$};
\draw (248,-38) node[anchor=north west,inner sep=0.75pt] {$t_{2}$};

\end{tikzpicture}
}}
\!\!\right)
+
\left(\!\!\!
\vcenter{\hbox{%
\tikzset{every picture/.style={line width=0.75pt}} 

\begin{tikzpicture}[x=0.75pt,y=0.75pt,yscale=0.7,xscale=0.7]
\path[use as bounding box] (245,-30) rectangle (420,-120);

\filldraw[fill=white,draw=black] (296.5,-90) circle (2pt);
\filldraw[fill=white,draw=black] (352,-90) circle (2pt);

\draw[dash pattern={on 0.80pt off 1pt},opacity=0.5]
  (352,-89) ++(-3.6:7pt)
  arc[start angle=360, end angle=115, radius=7pt];

\draw[
  decoration={markings, mark=at position 0.6 with {\arrow[scale=1.2]{stealth}}},
  postaction={decorate}
] (361,-90) -- (387.25,-64);

\draw[dash pattern={on 0.80pt off 1pt},opacity=0.5]
  (387.25,-64) -- (413.5,-38);

\draw[dash pattern={on 0.80pt off 1pt},opacity=0.5]
  (298,-87.5) -- (349,-38);

\draw[dash pattern={on 0.80pt off 1pt},opacity=0.5]
  (392,-38) -- (348,-81);

\filldraw[fill=black,draw=black] (387.25,-64) circle (1pt);
\draw (387.25,-64)
  node[anchor=north west,inner sep=0.75pt] {\small $t_{1}$};

\draw[dash pattern={on 0.30pt off 0.35pt}]
  (354,-87) -- (405.5,-35);

\draw[dash pattern={on 0.80pt off 1pt},opacity=0.5]
  (349,-38) -- (392,-38);

\draw (268.5,-38) -- (268.5,-59) -- (249,-59);

\draw (291.5,-105) node[anchor=north west,inner sep=0.75pt] {$0$};
\draw (347.5,-105) node[anchor=north west,inner sep=0.75pt] {$1$};
\draw (248,-38) node[anchor=north west,inner sep=0.75pt] {$t_{2}$};

\end{tikzpicture}
}}
\!\!\right)
\right\} \\
&=:\Delta_{\mathrm{\six}}+\Delta_{\mathrm{\seven}}+\Delta_{\mathrm{\eight}}+\Delta_{\mathrm{\nine}}+\Delta_{\mathrm{X}}+\Delta_{\mathrm{\eleven}}+\Delta_{\mathrm{\twelve}}
\end{aligned}
\end{equation}
\caption{Deformation and partition of the deformed domain}
\label{fig:deform00-2}
\end{figure}

\noindent
Consequently, we have 
	\begin{equation}\label{eq:V678910}
	V_{00}^{(0)}(x)=V_{\Delta_{\mathrm{\six}}}(x)+V_{\Delta_{\mathrm{\seven}}}(x)+V_{\Delta_{\mathrm{\eight}}}(x)+V_{\Delta_{\mathrm{\nine}}}(x)+V_{\Delta_{\mathrm{X}}}(x)+V_{\Delta_{\mathrm{\eleven}}}(x)+V_{\Delta_{\mathrm{\twelve}}}(x),
	\end{equation}
where
	\begin{equation}
	\begin{aligned}
	\Delta_{\mathrm{\six}}&=(t_1:0\to\infty e^{\sqrt{-1}\theta_{-1}})\times (t_2:0\to t_1), \\
	\Delta_{\mathrm{\seven}}&=(t_1:\infty e^{\sqrt{-1}\theta_{-1}}\to 1)\times (t_2:0\to \infty e^{\sqrt{-1}\theta_{-1}}), \\
	\Delta_{\mathrm{\eight}}&=(t_1:\infty e^{\sqrt{-1}\theta_{-1}}\to 1)\times (t_2:\infty e^{\sqrt{-1}\theta_{-1}}\to t_1), \\
	\Delta_{\mathrm{\nine}}&=(t_1:1\to\infty e^{\sqrt{-1}\theta_{-1}})\times (t_2:0\to \infty e^{\sqrt{-1}\theta_{-1}}), \\
	\Delta_{\mathrm{X}}&=(t_1:1\to\infty e^{\sqrt{-1}\theta_{-1}})\times (t_2:\infty e^{\sqrt{-1}\theta_{-1}}\to t_1), \\
	\Delta_{\mathrm{\eleven}}&=(t_1:1\to\infty e^{\sqrt{-1}\theta_{-1}})\times (t_2:t_1\to 1), \\
	\Delta_{\mathrm{\twelve}}&=(t_1:1\to\infty e^{\sqrt{-1}\theta_{-1}})\times (t_2:1\to t_1). 
	\end{aligned}
	\end{equation}
Note that $\Delta_{\mathrm{\six}}=\Delta_{00}^{\theta_{-1}}$, $\Delta_{\mathrm{\seven}}=\Delta_{\mathrm{\nine}}=\Delta_{\mathrm{10}}^{\theta_{-1}}$, $\Delta_{\mathrm{\eleven}}=\Delta_{\mathrm{\twelve}}=\Delta_{\mathrm{11}}^{\theta_{-1}}$, and $\Delta_{\mathrm{\eight}}=\Delta_{\mathrm{X}}=\Delta_{\mathrm{V}}$ as domains.
The branches of the integrand $\Phi(t_1,t_2;x)$ of $V_{\Delta_{\mathrm{\six}}}(x),\ldots,V_{\Delta_{\mathrm{\twelve}}}(x)$ on $\Delta_{\mathrm{\six}},\ldots,\Delta_{\mathrm{\twelve}}$ are given as follows: 
	\begin{align}
	&\Delta_{\mathrm{\six}}:
		\begin{cases}
\arg t_1=\theta_{-1}, \\
	\arg(1-t_1)=0 \quad \text{at $t_1=0$}, \\
	\arg t_2=\theta_{-1}, \\
	\arg(1-t_2)=0 \quad \text{at $t_2=0$}, \\
	\arg(t_1-t_2)=\theta_{-1},
		\end{cases} 
	\quad
	&\Delta_{\mathrm{\seven}}:
		\begin{cases}
		\arg t_1=0 \quad \text{at $t_1=1$}, \\
		\arg(1-t_1)=\theta_{-1}-\pi, \\
		\arg t_2=\theta_{-1}, \\
		\arg(1-t_2)=0\quad \text{at $t_2=0$}, \\
		\arg(t_1-t_2)=0\quad \text{at $(t_1,t_2)=(1,0)$},
		\end{cases} 
	\\
	&\Delta_{\mathrm{\eight}}:
		\begin{cases}
		\arg t_1=0 \quad \text{at $t_1=1$}, \\
		\arg(1-t_1)=\theta_{-1}-\pi, \\
		\arg t_2=0 \quad \text{at $t_2=1$}, \\
		\arg(1-t_2)=\theta_{-1}-\pi , \\
		\arg(t_1-t_2)=\theta_{-1}-\pi
		\end{cases} 
	\quad
	&\Delta_{\mathrm{\nine}}:
		\begin{cases}
		\arg t_1=0\quad \text{at $t_1=1$}, \\
		\arg(1-t_1)=\theta_{-1}+\pi, \\
		\arg t_2=\theta_{-1}, \\
		\arg(1-t_2)= 0 \quad\text{at $t_2=0$}, \\
		\arg(t_1-t_2)=0\quad \text{at $(t_1,t_2)=(1,0)$},
		\end{cases} 
	\\
	&\Delta_{\mathrm{X}}:
		\begin{cases}
		\arg t_1=0\quad \text{at $t_1=1$}, \\
		\arg(1-t_1)=\theta_{-1}+\pi, \\
		\arg t_2=0 \quad \text{at $t_2=1$}, \\
		\arg(1-t_2)=\theta_{-1}-\pi, \\
		\arg(t_1-t_2)=\theta_{-1}-\pi,
		\end{cases} 
	\quad
	&\Delta_{\mathrm{\eleven}}:
		\begin{cases}
		\arg t_1=0\quad \text{at $t_1=1$}, \\
		\arg(1-t_1)=\theta_{-1}+\pi, \\
		\arg t_2=0 \quad \text{at $t_2=1$}, \\
		\arg(1-t_2)=\theta_{-1}-\pi, \\
		\arg(t_1-t_2)=\theta_{-1},\phantom{\text{at $(t_1,t_2)=(1,0)$}}
		\end{cases} 
	\\
	&\Delta_{\mathrm{\twelve}}:
		\begin{cases}
		\arg t_1=0\quad \text{at $t_1=1$}, \\
		\arg(1-t_1)=\theta_{-1}+\pi, \\
		\arg t_2=0 \quad \text{at $t_2=1$}, \\
		\arg(1-t_2)=\theta_{-1}+\pi, \\
		\arg(t_1-t_2)=\theta_{-1}.
		\end{cases} 	
	\end{align}
Comparing these branches with \eqref{eq:branch1}, \eqref{eq:branch2} and \eqref{eq:branch3}, we have
	\begin{equation}
	\begin{aligned}
	&V_{\Delta_{\mathrm{\six}}}(x)=V_{00}^{(-1)}(x), 
	\quad
	V_{\Delta_{\mathrm{\seven}}}(x)=-e(-2b)V_{10}^{(-1)}(x), 
	\\
	&V_{\Delta_{\mathrm{\nine}}}(x)=V_{10}^{(-1)}(x),
	\quad 
	V_{\Delta_{\mathrm{\eleven}}}(x)=-e(-2b)V_{11}^{(-1)}(x),
	\quad 
	V_{\Delta_{\mathrm{\twelve}}}(x)=V_{11}^{(-1)}(x).
	\end{aligned}
	\end{equation}
Besides, comparing the branches and using Lemma~\ref{lem:DeltaV} yields
	\begin{equation}
	\begin{aligned}
	V_{\Delta_{\mathrm{\eight}}}(x)&=e(-2b)e(-2b)e(-2g)V_{\Delta_{\mathrm{V}}}(x)=e(-4b-g)V_{11}^{(-1)}(x)
	\\
	V_{\Delta_{\mathrm{X}}}(x)&=-e(-2b)e(-2g)V_{\Delta_{\mathrm{V}}}(x)=-e(-2b-g)V_{11}^{(-1)}(x).
	\end{aligned}
	\end{equation}
Substituting them into \eqref{eq:V678910}, we have
	\begin{equation}\label{eq:StokesV00}
	V_{00}^{(0)}(x)
	=V_{00}^{(-1)}(x)+\left(1-e(-2b)\right)V_{10}^{(-1)}(x)+\left(1-e(-2b)\right)\left(1-e(-2b-g)\right)V_{11}^{(-1)}(x).
	\end{equation}
	
Summarizing \eqref{eq:StokesV11}, \eqref{eq:StokesV10} and \eqref{eq:StokesV00}, we obtain the expression for $S_{-1}$.
For $S_0$, take $x\in \Pi_0$ and choose admissible angles $\theta_0\in(-\pi,0)$ and $\theta_1\in(-2\pi,-\pi)$ corresponding to $V_{ij}^{(0)}(x)$ and $V_{ij}^{(1)}(x)$, respectively.
Repeating the above deformation argument with the rotation of the angle from
$\theta_1$ to $\theta_0$, we obtain
	\begin{align}
	&V_{00}^{(1)}(x)=V_{00}^{(0)}(x), \\
	&V_{10}^{(1)}(x)=-\left(1-e(-2a-2b-g)\right)\left(1+e(g)\right)V_{00}^{(0)}(x)+V_{10}^{(0)}(x), \\
	&\begin{aligned}
	V_{11}^{(1)}(x)&=e(-g) \left(1-e(-2a-2b-g)\right)\left(1-e(-2a-2b-2g)\right)V_{00}^{(0)}(x)\\
	&\qquad-e(-g) \left(1-e(-2a-2b-2g)\right)V_{10}^{(0)}(x)+V_{11}^{(0)}(x).
	\end{aligned}
	\end{align}
These identities give the displayed formula for $S_0$.

Finally, the resulting identities between the fundamental solution
matrices extend to all parameters satisfying \eqref{eq:assumpStokes} by meromorphic continuation.
This completes the proof.

\section{Three aspects of the confluent Dotsenko--Fateev system}
\label{sec:threeaspects}

In this section, we examine three aspects of the confluent Dotsenko--Fateev system: an application to the Schr\"{o}dinger equation for quantum Painlev\'{e} V, realization as a restriction of a confluent $F_4$ system, and reduction to a rank-one equation.

\subsection{The Schr\"{o}dinger equation for quantum Painlev\'{e} V}
\label{sec:quantumPainleve}
H. Nagoya \cite{Nagoya2011} studied the Schr\"{o}dinger equations associated with the quantum Painlev\'{e} equations of type $\mathrm{J}=\mathrm{II},\ldots,\mathrm{VI}$.
In the case of $\mathrm{J}=\mathrm{V}$, the Schr\"{o}dinger equation is given by
	\begin{equation}\label{eq:Schrodinger}
	\hbar\frac{\p }{\p t}\Psi(X,t)=\widehat{H}_{\rm{V}}\left(X,\hbar\frac{\p}{\p X},t\right)\Psi(X,t),
	\end{equation}
where $\hbar\in\bbC\setminus\{0\}$ is a parameter and the Hamiltonian $\widehat{H}_{\rm{V}}$ is given by 
	\begin{equation}
	t\widehat{H}_{\rm{V}}=X(X-1)\left(\hbar\frac{\p}{\p X}\right)^2+\left(tX^2-(b_{\rm N}+c_{\rm N}+t)X+b_{\rm N}\right)\hbar\frac{\p}{\p X}+a_{\rm N}(b_{\rm N}+c_{\rm N}-a_{\rm N}+\hbar+t-tX),
	\end{equation}
where $a_{\rm N}, b_{\rm N}, c_{\rm N}\in\bbC$ are parameters. 
This Hamiltonian is obtained from the polynomial Hamiltonian of the fifth
Painlev\'{e} equation by replacing the canonical coordinates $q,p$ by
$X,\hbar\partial/\partial X$, respectively. 

We shall show that the confluent Dotsenko--Fateev system is equivalent to a finite-dimensional reduction of the Schr\"{o}dinger equation \eqref{eq:Schrodinger}. 
For $m\in\bbZ_{\geq 0}$, set
\[
 \mathcal{P}_m:=\bigoplus_{i=0}^{m}\bbC X^i.
\]
By Proposition 2.3 of \cite{Nagoya2011}, the subspace $\mathcal P_m$ is invariant under $\widehat{H}_{\rm V}$ if and only if $a_{\rm N}=m\hbar$.
In this case, the Schr\"{o}dinger equation \eqref{eq:Schrodinger} can be restricted to $\mathcal P_m$.
	Equivalently, writing
	\begin{equation}\label{eq:SchrodingerPolySol}
    	\Psi(X,t)=\sum_{i=0}^{m}\psi_i(t)X^i, 
	\end{equation}
the restricted equation is identified with an $(m+1)$-dimensional linear differential system for the coefficient vector $\psi(t)=(\psi_0(t),\ldots,\psi_m(t))^{T}$. 
We now specialize to $m=2$. 
\begin{proposition}\label{prop:NagoyacDF}
Suppose that $m=2$ and $a_{\rm{N}}=2\hbar$. 
Then the three-dimensional linear differential system obtained by restricting the Schr\"{o}dinger equation \eqref{eq:Schrodinger} to the invariant polynomial subspace $\mathcal P_2$ is equivalent to the confluent Dotsenko--Fateev system \eqref{eq:cDFeq} by
\begin{equation}\label{eq:correspSchrodingertoDF}
	t=x,\quad
	\hbar=\frac{g}{2}, \quad 
	b_{\rm{N}}=a+b+g, \quad
	c_{\rm{N}}=-b.
\end{equation}
\end{proposition}
\begin{proof}
We first calculate the action of $\widehat{H}_{\rm V}$ on the invariant subspace $\mathcal{P}_{2}$. 
With respect to the basis $\{1,X,X^2\}$, a direct computation gives
	\begin{equation}
	\left(\widehat{H}_{\rm V}1,\widehat{H}_{\rm V}X,\widehat{H}_{\rm V}X^2\right)
	=\hbar(1,X,X^2)M(t),
	\end{equation}
where
	\begin{equation}
	M(t)=\begin{pmatrix}
	\frac{2(b_{\rm N}+c_{\rm N}-\hbar+t)}{t} & \frac{b_{\rm N}}{t} & 0 \\
	-2 & \frac{b_{\rm N}+c_{\rm N}-2\hbar+t}{t} & \frac{2(b_{\rm{N}}-\hbar)}{t}\\
	0 &  -1  & 0
	\end{pmatrix}.
	\end{equation}
This implies that the restriction of the Schr\"{o}dinger equation \eqref{eq:Schrodinger} to the invariant polynomial subspace $\mathcal P_2$ is 
	\begin{equation}\label{eq:SchrodingerRest}
	\frac{d\psi}{dt}=M(t)\psi,
	\end{equation}
where $\psi(t)=(\psi_0(t),\psi_1(t),\psi_2(t))^{T}$. 
Setting
	\begin{equation}\label{eq:SchrodingerP}
	\psi=P_{\rm{N}}\tilde{\psi},
	\quad 
	P_{\rm{N}}=\begin{pmatrix}
	0 & 0 & -1 \\
	0 & -1 & 2 \\
	-1 & 1 & -1
	\end{pmatrix},
	\end{equation} 
the system \eqref{eq:SchrodingerRest} is transformed into 
	\begin{equation}\label{eq:Schrodingerreduction}
	\frac{d\widetilde{\psi}}{dt}=(P_{\rm{N}}^{-1}M(t)P_{\rm{N}})\widetilde{\psi},
	\end{equation}
where
	\begin{equation}
	P_{\rm{N}}^{-1}M(t)P_{\rm{N}}=\begin{pmatrix}
	0 &  & \\
	& 1 & \\
	& & 2
	\end{pmatrix}
	+
	\frac{1}{t}
	\begin{pmatrix}
	-2\hbar+2b_{\rm{N}} & c_{\rm{N}} & 0 \\
	-2\hbar+2b_{\rm{N}} & b_{\rm{N}}+c_{\rm{N}} & -2\hbar+2c_{\rm{N}}\\
	0 & b_{\rm{N}} & -2\hbar+2c_{\rm{N}}
	\end{pmatrix}.
	\end{equation}
By setting the parameters as in \eqref{eq:correspSchrodingertoDF}, this system coincides with the confluent Dotsenko--Fateev system \eqref{eq:cDFeq}. 
\end{proof}
Nagoya \cite{Nagoya2011} also constructed polynomial solutions of the form \eqref{eq:SchrodingerPolySol} to the Schr\"{o}dinger equation \eqref{eq:Schrodinger}. 
In other words, he gave solutions to the finite-dimensional reduction of \eqref{eq:Schrodinger} to $\mathcal{P}_m$. 
We now relate our results to Nagoya's work. 
For this purpose, we recall his integral formula in the case $m=2$.
\begin{theorem}[The case of $\rm{J}=\rm{V}$ and $m=2$ of Theorem 2.5 of \cite{Nagoya2011}]
The integral 
	\begin{equation}\label{eq:SchrodingerPoly}
	\Psi_{\Delta}(X,t)=\int_{\Delta} (u_1-u_2)^{2\hbar}\prod_{i=1}^{2}u_i^{-b_{\rm{N}}-1}(1-u_i)^{-c_{\rm{N}}-1}e^{u_i t}(X-u_i)\,du_1\wedge du_2
	\end{equation}
gives a solution to the Schr\"{o}dinger equation \eqref{eq:Schrodinger} for the quantum Painlev\'{e} equation of type {\rm V} with $a_{\rm{N}}=2\hbar$, where $\Delta$ is an appropriate cycle. 	
\end{theorem}
We refer to solutions $\Psi(X,t)$ of \eqref{eq:Schrodinger} such that $\Psi(\cdot,t)\in\mathcal P_2$ for each $t$ as \emph{polynomial solutions in $\mathcal P_2$}.
In particular, $\Psi_\Delta(X,t)$ is such a solution.
To ensure convergence of the integrals, we assume for the moment that
	\begin{equation}
	\re(a_{\rm N})>0, \quad \re(b_{\rm N}),~\re(c_{\rm N})<0.
	\end{equation}
Under the parameter correspondence \eqref{eq:correspSchrodingertoDF}, these conditions are equivalent to \eqref{eq:assumpconv}.
For comparison with Nagoya's integral formula, we rewrite the integral \eqref{eq:cDFint} using $u_1,u_2$ as the integration variables:
	\begin{equation}
	V_{\Delta}(x)=\int_{\Delta}(u_1-u_2)^{g}\prod_{i=1}^{2}u_i^{-(a+b+g)}(1-u_i)^{b}e^{xu_i}	\vec{\varphi}
	\end{equation}
and $\vec{\varphi}=(\varphi_1,\varphi_2,\varphi_3)^T$ is given by
	\begin{equation}
	\begin{aligned}
	\varphi_1&=\frac{du_1\wedge du_2}{u_1u_2}, \\
	\varphi_2&=-\frac{du_1\wedge du_2}{u_1(1-u_2)}-\frac{du_1\wedge du_2}{u_2(1-u_1)}, \\
	\varphi_3&=\frac{du_1\wedge du_2}{(1-u_1)(1-u_2)}.
\end{aligned}
	\end{equation}
Let $P_{\rm{N}}$ denote the constant matrix \eqref{eq:SchrodingerP} appearing in the proof of Proposition~\ref{prop:NagoyacDF}. 
A direct computation gives
	\begin{equation}\label{eq:Nagoya-cDF-integrand}
 	-(1,X,X^2)P_{\rm{N}}\vec{\varphi}
	=
 	\frac{(X-u_1)(X-u_2)}{u_1u_2(1-u_1)(1-u_2)}
	 \,du_1\wedge du_2.
	\end{equation}
Hence, under the parameter correspondence \eqref{eq:correspSchrodingertoDF}, we obtain
	\begin{equation}\label{eq:Nagoya-cDF-solution}
	 \Psi_{\Delta}(X,t)
 	=
 	-(1,X,X^2)P_{\rm{N}}V_{\Delta}(t),
	\end{equation}
provided that the same cycle and branches are chosen on both sides.
By using this correspondence, we can construct a fundamental system of solutions of the degree-two polynomial reduction of the Schr\"{o}dinger equation \eqref{eq:Schrodinger}, together with the associated Stokes matrices. 
For $k\in\bbZ$, we set
	\begin{equation}
	\Psi_{ij}^{(k)}(X,t):=-(1,X,X^2)P_{\rm{N}}V_{ij}^{(k)}(t)\quad (ij\in\{00,10,11\}).
	\end{equation}
Then $\Psi_{00}^{(k)}(X,t)$, $\Psi_{10}^{(k)}(X,t)$ and $\Psi_{11}^{(k)}(X,t)$ are defined on $\Sigma_{k}$ and form a basis of the space of polynomial solutions in $\mathcal{P}_2$.
This follows from the fact that $\mathcal{V}^{(k)}(t)$ is a fundamental solution matrix on $\Sigma_{k}$.  
As explained in Section~\ref{sec:SecandAsy},
the matrix $\calV^{(k)}(t)$ admits a meromorphic continuation in $(a,b,g)$.
Through \eqref{eq:Nagoya-cDF-solution}, this induces a
meromorphic continuation of $\Psi_{ij}^{(k)}(X,t)$ in the parameters.
We use the same notation for these continuations.
The continued solutions form a basis of the space of
polynomial solutions in $\mathcal P_2$ under
	\begin{equation}\label{eq:assumpStokesSch}
	\begin{aligned}
	b_{\rm N},~~c_{\rm N},~~
	b_{\rm N}-\frac{a_{\rm N}}{2},~~
	c_{\rm N}-\frac{a_{\rm N}}{2}
	\notin\mathbb Z_{\geq0},\quad	a_{\rm N}\notin\{-1,-3,-5,\ldots\}.
	\end{aligned}
	\end{equation}
These conditions are equivalent to \eqref{eq:assumpStokes}
under the parameter correspondence
\eqref{eq:correspSchrodingertoDF}.

From Theorem~\ref{thm:StokesMat} and the correspondence \eqref{eq:correspSchrodingertoDF}, we immediately obtain the following.
\begin{corollary}\label{cor:StokesSch}
Assume $a_{\rm{N}}=2\hbar\neq0$ and \eqref{eq:assumpStokesSch}.
For each $k\in \mathbb{Z}$, it holds that
	\begin{equation}
	\left(
	\Psi_{00}^{(k+1)}(X,t), \Psi_{10}^{(k+1)}(X,t), \Psi_{11}^{(k+1)}(X,t)
	\right)
	=\left(
	\Psi_{00}^{(k)}(X,t), \Psi_{10}^{(k)}(X,t), \Psi_{11}^{(k)}(X,t)
	\right)S_{k}
	\end{equation}
on $\Sigma_{k}\cap\Sigma_{k+1}$. 
Here $\{S_k\}_{k\in\bbZ}$ are the Stokes matrices determined by
	\begin{align}
	&S_{-1}=\begin{pmatrix}
	1 & 0 & 0 \\
	1-e(2c_{\rm{N}}) & 1 & 0 \\
	\left(1-e(2c_{\rm{N}}))(1-e(2c_{\rm{N}}-2\hbar)\right) & \left(1-e(2c_{\rm{N}}-2\hbar)\right)\left(1+e(2\hbar)\right) & 1
	\end{pmatrix}, 
	\\
	&S_{0}=\begin{pmatrix}
	 1 & -\left(1-e(2\hbar-2b_{\rm{N}})\right)\left(1+e(2\hbar)\right) &
  e(-2\hbar) \left(1-e(2\hbar-2b_{\rm{N}})\right)\left(1-e(-2b_{\rm{N}})\right) \\
 	0 & 1 & -e(-2\hbar) \left(1-e(-2b_{\rm{N}})\right) \\
 	0 & 0 & 1 
	\end{pmatrix}
	\end{align}
and the relation \eqref{eq:Stokespreiod}, where
	$
	\Lambda=
	\operatorname{diag}
	\left[
	2b_{\rm N}-2\hbar,\,
	b_{\rm N}+c_{\rm N},\,
	2c_{\rm N}-2\hbar
	\right]
	$.
\end{corollary}
Nagoya's integral formula \cite[Theorem 2.5]{Nagoya2011} is available for arbitrary $m\in\mathbb Z_{\geq0}$, whereas the present identification with the confluent Dotsenko--Fateev system concerns the case $m=2$.
It would be interesting to investigate whether analogous explicit Stokes analysis can be carried out for higher-degree polynomial reductions.

\subsection{Confluence and restriction of the $F_4$ system}

Let us show the relation between the confluent Dotsenko--Fateev system and the confluent $F_4$ system. 
We first recall the relation between the Dotsenko--Fateev equation and the linear Pfaffian system associated with Appell's hypergeometric function $F_4$. 
Throughout this subsection, the parameters appearing below are assumed to be generic.
Appell's hypergeometric function $F_4(\alpha,\beta;\gamma,\gamma';x,y)$ is defined by
	\begin{equation}\label{eq:F4}
	F_4(\alpha,\beta;\gamma,\gamma';x,y)=\sum_{m,n=0}^{\infty} \frac{(\alpha)_{m+n}(\beta)_{m+n}}{m!n!(\gamma)_m (\gamma')_n}x^m y^n,
	\end{equation}
where $(\alpha)_n=\Gamma(\alpha+n)/\Gamma(\alpha)$. 
It is known that this function satisfies a rank-four Pfaffian system in two variables.  
M. Kato \cite{Kato1988} derived the rank-four Pfaffian system 
	\begin{equation}\label{eq:F4Pfaff}
	dU=\left(B_1\frac{dx}{x}+B_2\frac{dx}{x-1}+B_3\frac{d(x-y)}{x-y}+B_4\frac{dy}{y}+B_5\frac{dy}{y-1}\right)U
	\end{equation}
as a pull-back of the Pfaffian system for $F_4$, where $B_1,\ldots,B_5 \in \mathrm{Mat}(4,\bbC)$. 
Specifically, the system \eqref{eq:F4Pfaff} is associated with $F_4(\alpha,\beta;\gamma,\gamma';xy,(x-1)(y-1))$. 
The Pfaffian system \eqref{eq:F4Pfaff} has logarithmic singularities along the hyperplanes
	\begin{equation}
	\{x=0\}\cup\{x=1\}\cup\{x=y\}\cup\{y=0\}\cup\{y=1\}.
	\end{equation}
Here the residue matrix $B_3$ at $x=y$ satisfies $B_3\sim\operatorname{diag}[0,0,0,2(\gamma+\gamma'-\alpha-\beta-1)]$. 
This implies that the restriction of the system \eqref{eq:F4Pfaff} to the singular locus $x=y$ is an ordinary differential equation of rank three. 
Here the ``restriction" refers to the differential equation satisfied by the solutions of the system \eqref{eq:F4Pfaff} which are holomorphic at $x=y$. 
Kato \cite{Kato1995} derived this restricted system explicitly, and studied the connection problem and monodromy group. 
After that, Haraoka \cite{Haraoka2013} pointed out that the restricted system coincides with the Dotsenko--Fateev system \eqref{eq:DFeq} by setting 
	\begin{equation}
	\begin{aligned}
	\alpha=-c,\quad \beta=-a-b-c-\frac{g}{2}, \quad
	\gamma=1-a-c-\frac{g}{2},\quad \gamma'=-b-c-\frac{g}{2}.
	\end{aligned}
	\end{equation}

Then it is natural to ask whether this relation is compatible with the confluence considered in Section~\ref{sec:conf}. 
Haraoka \cite{Haraoka2024} considered the confluence for the Pfaffian system of $F_4$. 
He first used middle convolution to construct the Pfaffian system of rank four, which is equivalent to the one associated with
	\begin{equation}
	y^{\alpha_1+\alpha_3}(y-1)^{\alpha_2+\alpha_3}F_4(\alpha_3-\beta,\alpha_1+\alpha_2+\alpha_3;1+\alpha_1+\alpha_3,\alpha_2+\alpha_3;xy,(y-1)(x-1)).
	\end{equation}
For that Pfaffian system, he then considered the confluence
	\begin{equation}\label{eq:Haraokaconf}
	\left(\alpha_1,\alpha_3,x,y\right)\mapsto (\alpha_1-\gamma\varepsilon^{-1},\gamma\varepsilon^{-1},\varepsilon x,\varepsilon y),\quad \varepsilon\to 0, 
	\end{equation}
which makes $\{x=1\}\to\{x=\infty\}$ and $\{y=1\}\to\{y=\infty\}$. 
The resulting confluent $F_4$ system is given by 
	\begin{equation}\label{eq:HaraokaPfaff}
	dU=\left(C_1\frac{dx}{x}+C_2\,dx+C_3\frac{d(x-y)}{x-y}+C_4\frac{dy}{y}+C_5\,dy\right)U
	\end{equation}
with 
	\begin{equation}
	\begin{aligned}
	C_1&=\begin{pmatrix}
	\alpha_2 & \alpha_1+\alpha_2 & 0 & 0 \\
	-\alpha_2 & -\alpha_1-\alpha_2 & 0 & 0 \\
	0 & 0 & \alpha_2+\beta & \alpha_1+\alpha_2+\beta \\
	0 & 0 & -\alpha_2-\beta & -\alpha_1-\alpha_2-\beta
	\end{pmatrix}, 
	\quad
	C_2=\begin{pmatrix}
	\gamma & & & \\
	& 0 & & \\
	& & \gamma & \\
	& & & 0
	\end{pmatrix},
	\\
	C_3&=\begin{pmatrix}
	\beta & 0 & 0 & \beta \\
	0 & 0 & 0 & 0\\
	0 & 0 & 0 & 0\\
	\beta & 0 & 0 & \beta
	\end{pmatrix},
	\quad
	C_4=\begin{pmatrix}
	-\alpha_2-\beta & 0 &\alpha_2+\beta & 0 \\
	0 & -\alpha_2 & 0 & \alpha_2 \\
	-\alpha_1-\alpha_2-\beta & 0 & \alpha_1+\alpha_2+\beta & 0 \\
	0& -\alpha_1-\alpha_2 & 0 & \alpha_1+\alpha_2
	\end{pmatrix},
	\\
	C_5&=\begin{pmatrix}
	-\gamma & & & \\
	& -\gamma & & \\
	& & 0 & \\
	& & & 0
	\end{pmatrix}. 
	\end{aligned}
	\end{equation}
This system has logarithmic singularities along $\{x=0\}\cup\{x=y\}\cup\{y=0\}$ and irregular singularities along $\{x=\infty\}\cup\{y=\infty\}$. 
We now show that, after a scalar gauge transformation,
the restriction of \eqref{eq:HaraokaPfaff} to $x=y$
is equivalent to the confluent Dotsenko--Fateev system
\eqref{eq:cDFeq}.
For this purpose, we set
\begin{equation}\label{eq:HaraokaScalarGauge}
   \widetilde{U}(x,y)=y^{-\alpha_1}e^{\gamma y}U(x,y).
\end{equation}
Then $\widetilde{U}(x,y)$ satisfies
\begin{equation}\label{eq:HaraokaPfaffgauge}
    d\widetilde{U}=
    \left(
    C_1\frac{dx}{x}+C_2\,dx
    +C_3\frac{d(x-y)}{x-y}
    +(C_4-\alpha_1 I)\frac{dy}{y}
    +(C_5+\gamma I)\,dy
    \right)\widetilde{U}.
\end{equation}
We consider the restriction of this system to the singular locus $x=y$. 
\begin{proposition}\label{prop:restcF4}
	The restriction of the Pfaffian system \eqref{eq:HaraokaPfaffgauge} to the singular locus $x=y\,(=:z)$ is given by the system of linear ordinary differential equations 
	\begin{equation}\label{eq:restriction}
	\frac{dV}{dz}=\left(D_{\infty}+\frac{D_{1}}{z}\right)V,
	\end{equation}
	where
	\begin{equation}
	D_{\infty}=\begin{pmatrix}
	0 & & \\
	& \gamma & \\
	 &  & 2\gamma
	\end{pmatrix},
	\quad
	D_{1}=\begin{pmatrix}
	-2 \left(\alpha _1+\alpha _2\right) & \alpha _2 & 0 \\
	 -2 \left(\alpha _1+\alpha _2\right) & -\alpha _1-\beta  & 2 \left(\alpha _2+\beta
   \right) \\
 	0 & -\alpha _1-\alpha _2-\beta  & 2 \left(\alpha _2+\beta \right)
	\end{pmatrix}.
	\end{equation}
	Moreover, this equation coincides with the confluent Dotsenko--Fateev system \eqref{eq:cDFeq} by 
	\begin{equation}\label{eq:correspF4toDF}
	z=\frac{x}{\gamma},\quad 
	\alpha_1=-a-\frac{g}{2},\quad 
	\alpha_2=-b,\quad 
	\beta=-\frac{g}{2}.
	\end{equation}
\end{proposition}
\begin{proof}
	Set 
	\begin{equation}
	\widetilde{U}(x,y)=P\widetilde{V}(x,y),
	\quad
	P=
	\begin{pmatrix}
	0 & -\frac{1}{2} & 0 & 1 \\
 	1 & 0 & 0 & 0 \\
 	0 & 0 & -1 & 0 \\
 	0 & \frac{1}{2} & 0 & 1
	\end{pmatrix}.
	\end{equation}
Then the residue matrix $C_3$ at $x=y$ becomes the diagonal matrix $P^{-1}C_3P=\operatorname{diag}[0,0,0,2\beta]$ and we have
	\begin{align}
	&P^{-1}\left(C_2+(C_5+\gamma I)\right)P=
	\left(
	\begin{array}{ccc|c}
	& & & 0 \\
	& D_{\infty} & & 0\\
	& & & 0 \\ \hline 
	0 & 0 & 0 & \gamma
	\end{array}
	\right) 
	\\
	&P^{-1}\left(C_1+(C_4-\alpha_1 I)\right)P=
	\left(
	\begin{array}{ccc|c}
	& & & 0 \\
	& D_{1} & & 0\\
	& & & 0 \\ \hline 
	0 & 0 & 0 & -\beta-\alpha_1
	\end{array}
	\right).
	\end{align}
This shows that the restriction is given by the system \eqref{eq:restriction}. 
Under the substitutions \eqref{eq:correspF4toDF}, the system \eqref{eq:restriction} coincides with the confluent Dotsenko--Fateev system \eqref{eq:cDFeq}.
\end{proof}
\begin{remark}
Here we explain where the scalar gauge transformation \eqref{eq:HaraokaScalarGauge} comes from.
Comparing Kato's $F_4$ system \eqref{eq:F4Pfaff} with Haraoka's
$F_4$ system obtained by middle convolution, one finds that, up to the corresponding parameter identification, the two systems are related by the scalar gauge factor
	\begin{equation}
 	y^{\alpha_1+\alpha_3}(y-1)^{\alpha_2+\alpha_3}.
	\end{equation}
Under Haraoka's confluence \eqref{eq:Haraokaconf}, this factor tends, up to a nonzero factor independent of $x$ and $y$, to
	\begin{equation}
 	y^{\alpha_1}e^{-\gamma y}.
	\end{equation}
Hence the scalar gauge transformation \eqref{eq:HaraokaScalarGauge} used above precisely cancels this factor after confluence.
\end{remark}

Thus, modulo scalar gauge transformations, the relations above are summarized schematically by the following diagram.
\begin{equation}
\begin{tikzcd}[column sep=large,row sep=large]
\text{Pull-back of the }F_4\text{ system \eqref{eq:F4Pfaff}} 
  \arrow[d, "\operatorname{res}_{x=y}"']
  \arrow[r, "\mathrm{confluence}"]
&
\text{Confluent }F_4\text{ system \eqref{eq:HaraokaPfaff}}
  \arrow[d, "\operatorname{res}_{x=y}"]
\\
\text{Dotsenko--Fateev system \eqref{eq:DFeq}}
  \arrow[r, "\mathrm{confluence}"]
&
\text{Confluent Dotsenko--Fateev system \eqref{eq:cDFeq}}.
\end{tikzcd}
\end{equation}
Here $\mathrm{res}_{x=y}$ denotes the restriction to the singular locus $x=y$. 

\begin{remark}\label{rem:AdaHara}
The restriction in Proposition~\ref{prop:restcF4} does not
immediately yield the Stokes matrices in
Theorem~\ref{thm:StokesMat} from Haraoka's formulas
\cite{Haraoka2024}. 
Indeed, Haraoka used the blowing-up $y=\zeta x$ to study
the asymptotic behavior of the system \eqref{eq:HaraokaPfaff}
near $(x,\zeta)=(\infty,\infty)$
\cite[Proposition~3.2]{Haraoka2024}.
In these coordinates, the diagonal $x=y$ corresponds to $\zeta=1$, so the point at infinity of the restricted system corresponds to $(x,\zeta)=(\infty,1)$.
It would be interesting to study the asymptotic behavior of the confluent $F_4$ system near $(x,\zeta)=(\infty,1)$ and clarify the relation between its Stokes data and those of the restricted system.
\end{remark}

\subsection{Reduction to rank one}
\label{sec:reduction}

Finally, we show that the confluent Dotsenko--Fateev system \eqref{eq:cDFeq} can be reduced to a trivial rank-one (scalar) equation by combining the middle Laplace transform, middle convolution, gauge transformations and a quadratic change of the independent variable.
For the definitions of the middle Laplace transform and middle convolution, see \cite{adachi2025arXiv} and \cite{DR2007}, respectively.

\begin{theorem}\label{thm:reduction}
For generic parameters $(a,b,g)$, the confluent Dotsenko--Fateev system \eqref{eq:cDFeq} can be reduced to
the trivial rank-one equation
	\begin{equation}\label{eq:trivial}
    \frac{d\tilde{z}}{dt}=0
	\end{equation}
by combining the middle Laplace transform, middle convolution, gauge transformations and a quadratic change of the independent variable.
An explicit sequence of transformations is given in the proof.
\end{theorem}

\begin{proof}
We first apply the middle Laplace transform to the confluent Dotsenko--Fateev system \eqref{eq:cDFeq}.
This yields the rank-two Fuchsian system 
\begin{equation}\label{eq:cDFtoFuchs}
	\frac{dW}{dx}=\left(\frac{B_1}{x}+\frac{B_2}{x-1}+\frac{B_3}{x-2}\right)W,
\end{equation}
where 
	\begin{equation}
	B_1=\begin{pmatrix}
	-2a-2b-g & 0 \\
	0 & 0
	\end{pmatrix},
	~
	B_2=\begin{pmatrix}
	b & b \\
	-a-b-g & -a-b-g
	\end{pmatrix},
	~
	B_3=\begin{pmatrix}
	0 & 0 \\
	0 & 2b+g 
	\end{pmatrix}
	\end{equation}
and the residue matrix $B_{\infty}:=-(B_1+B_2+B_3)$ at $x=\infty$ is given by
	\begin{equation}
	B_{\infty}=\begin{pmatrix}
	 2 a+b+g & -b \\
 	a+b+g & a-b 
 	\end{pmatrix}.
	\end{equation}
We next make the gauge transformation
	\begin{equation}
	W(x)=(x-1)^{-a-\frac{g}{2}}\begin{pmatrix}1 & 0 \\0  & \frac{x}{x-2}\end{pmatrix}\widetilde{W}(x).
	\end{equation}
Then we have
	\begin{equation}\label{eq:tildeWeq}
	\frac{d\widetilde{W}}{dx}=\left(\frac{\widetilde{B}_1}{x}+\frac{\widetilde{B}_2}{x-1}+\frac{\widetilde{B}_3}{x-2}\right)\widetilde{W},
	\end{equation}
where
	\begin{equation}
	\widetilde{B}_1=\begin{pmatrix}
	-2a-2b-g & 0 \\
	-2a-2b-2g & -1
	\end{pmatrix},
	~~
	\widetilde{B}_{2}=\begin{pmatrix}
	a+b+\frac{g}{2} & -b \\
	a+b+g & -b-\frac{g}{2}
	\end{pmatrix},
	~~ 
	\widetilde{B}_3=\begin{pmatrix}
	0 & 2b \\
	0 & 2b+g+1
	\end{pmatrix}
	\end{equation}
and the residue matrix $\widetilde{B}_{\infty}$ at $x=\infty$ is given by
	\begin{equation}
	\widetilde{B}_{\infty}=-(\widetilde{B}_1+\widetilde{B}_2+\widetilde{B}_3)=
	\begin{pmatrix}
	 a+b+\frac{g}{2} & -b \\
 	a+b+g & -b-\frac{g}{2}
 	\end{pmatrix}=\widetilde{B}_2.
	\end{equation}
The quadratic change of variable
	\begin{equation}\label{eq:quadratic}
	\widetilde{W}(x)=Z(t),\quad 
	t=\frac{x^2}{4(x-1)}
	\end{equation}
transforms the system \eqref{eq:tildeWeq} into 
	\begin{equation}\label{eq:GaussSystem}
	\frac{dZ}{dt}=\left(\frac{C_1}{t}+\frac{C_2}{t-1}\right)Z,
	\end{equation}
where 
	\begin{equation}
	C_{1}=\frac{1}{2}\widetilde{B}_1=\begin{pmatrix}
	-a-b-\frac{g}{2} & 0 \\
	-a-b-g & -\frac{1}{2}
	\end{pmatrix},
	\quad 
	C_2=\frac{1}{2}\widetilde{B}_3=\begin{pmatrix}
	0 & b \\
	0 & b+\frac{g+1}{2}
	\end{pmatrix}.
	\end{equation}
To apply the middle convolution so as to reduce the rank, we first make the scalar gauge transformation
	\begin{equation}
	Z(t)=t^{-a-b-\frac{g}{2}}\widetilde{Z}(t).
	\end{equation}
Then the system \eqref{eq:GaussSystem} is transformed into
	\begin{equation}\label{eq:GaussSystemGauge}
	\frac{d\widetilde{Z}}{dt}=\left(\frac{\widetilde{C}_1}{t}+\frac{\widetilde{C}_2}{t-1}\right)\widetilde{Z},
	\end{equation}
where  
	\begin{equation}
	\widetilde{C}_{1}=C_1+\left(a+b+\frac{g}{2}\right)I=\begin{pmatrix}
	0 & 0 \\
	-a-b-g &a+b+\frac{g-1}{2}
	\end{pmatrix}
	\end{equation}
and $\widetilde{C}_2=C_2$. 
Note that the residue matrix $\widetilde{C}_{\infty}$ at $t=\infty$ is given by
	\begin{equation}
	\widetilde{C}_{\infty}=\begin{pmatrix}
	 0 & -b \\
 	a+b+g & -a-2b-g
	\end{pmatrix}
	\end{equation}
and its eigenvalues are $-b$ and $-a-b-g$. 
Then, applying the middle convolution with the parameter $-b$ to the system \eqref{eq:GaussSystemGauge} yields the rank-one system (scalar equation)
	\begin{equation}
	\frac{dz}{dt}=\left(\frac{2a+g-1}{2t}+\frac{g+1}{2(t-1)}\right)z.
	\end{equation}
Lastly, by setting 
	\begin{equation}
	z(t)=t^{a+\frac{g-1}{2}}(t-1)^{\frac{g+1}{2}}\tilde{z}(t),
	\end{equation}
we arrive at the trivial rank-one equation \eqref{eq:trivial}.
\end{proof}

\begin{remark}
Scalarization of a system of linear ordinary differential equations may in general introduce apparent singularities. 
However, both rank-two systems \eqref{eq:cDFtoFuchs} and \eqref{eq:GaussSystemGauge} appearing above can be scalarized without introducing
apparent singularities.

First, for the system \eqref{eq:cDFtoFuchs}, set $\xi=x-1$ and $W=(w_1,(1-\xi)w_2)^T$. 
Then we see that $w_2$ satisfies
	\begin{equation}
	w_2''+
	\left(
	\frac{2a+2b+g}{\xi+1}
	+
	\frac{a+g+1}{\xi}
	+
	\frac{2-2b-g}{\xi-1}
	\right)w_2'+\dfrac{Q(\xi)}{\xi(\xi-1)(\xi+1)}
	w_2=0
	\end{equation}
with $Q(\xi)=(2a+g+1)(a+1)\xi+(2a+g+1)(1-a-2b-g)$. 
This equation is nothing but the Heun differential equation with special parameters. 

Similarly, write $\widetilde{Z}=(z_1,z_2)^T$ in \eqref{eq:GaussSystemGauge}. 
Eliminating $z_2$, we obtain
	\begin{equation}
	t(1-t)z_1''+\left\{-a-b-\frac{g-1}{2}-(-a-2b-g+1)t\right\}z_1'-b(a+b+g)z_1=0.
	\end{equation}
This is nothing but the Gauss hypergeometric equation associated with ${}_2F_1(-b,-a-b-g;-a-b-\frac{g-1}{2};t)$. 
Thus, at the scalar level, the quadratic transformation \eqref{eq:quadratic} appearing
above may be regarded as a reduction of a special Heun equation to
a Gauss hypergeometric equation.
For related quadratic reductions of Heun equations
to the Gauss hypergeometric equation, see Maier \cite{Maier2005}.
\end{remark}

We conclude with a brief observation from the viewpoint of Katz's theory (cf.~\cite{HaraokaBook}). 
The rigidity indices of the systems \eqref{eq:cDFeq} and
\eqref{eq:cDFtoFuchs} are both $0$.
Thus, in the present example, the middle Laplace transform does not
change the rigidity index.
By contrast, the rigidity index of the system \eqref{eq:GaussSystem}
obtained after the quadratic change of the independent variable is $2$.
The latter system can then be reduced to the trivial rank-one equation
by middle convolution and scalar gauge transformations.
Thus the quadratic change of the independent variable is the step at
which the rigidity index changes in the above reduction.

At the same time, the effect of this quadratic transformation on the
global behavior of solutions can also be traced explicitly.
This suggests that it may be useful to combine Katz-theoretic operations
with suitable algebraic transformations of the independent variable,
both for the reduction and for the global analysis of linear differential
equations.

\section*{Acknowledgements}
The author would like to thank Hajime Nagoya for bringing his related work \cite{Nagoya2011} on the Schr\"{o}dinger equations for the quantum Painlev\'{e} equations to the author's attention and for several subsequent helpful discussions and explanations concerning the relation of that work to the present study.
The author also thanks Akihito Ebisu for indicating that the Fuchsian system \eqref{eq:tildeWeq} can be reduced to the Gauss hypergeometric system by a quadratic change of variable. 
The author also thanks Saiei-Jaeyeong Matsubara-Heo for suggesting the use of the dominated convergence theorem to improve the proofs of Proposition~\ref{prop:firstterm} and Theorem~\ref{thm:asymptoticexpansion}.

The author used ChatGPT (OpenAI) to assist with language editing and with the presentation and checking of some mathematical arguments. The author has independently verified the mathematical content and takes full responsibility for the paper.

This work was partially supported by JSPS KAKENHI Grant Number 24K22826. 
This work was also supported by the Research Institute for Mathematical
Sciences, an International Joint Usage/Research Center located in Kyoto
University.


\bibliographystyle{plain}
\bibliography{references}

\medskip\noindent
Shunya Adachi\\*
Cooperative Faculty of Education\\*
Utsunomiya University\\*
350 Mine-machi, Utsunomiya-City, Tochigi, 321-8505, Japan\\*
{\tt sadachi@a.utsunomiya-u.ac.jp}

\end{document}